\documentclass[a4paper,10pt]{article}%
\usepackage{amsmath}%
\usepackage{amsfonts}%
\usepackage{amssymb}%
\usepackage{amsthm}
\usepackage{stmaryrd}
\usepackage{graphicx}
\usepackage[latin1]{inputenc}
\usepackage[T1]{fontenc}
\usepackage[english,french]{babel}

\newtheorem{theorem}{Theorem}[section]

\newtheorem{conjecture}[theorem]{Conjecture}
\newtheorem{corollary}[theorem]{Corollary}

\newtheorem{definition}[theorem]{Definition}

\newtheorem{lemma}[theorem]{Lemma}

\newtheorem{proposition}[theorem]{Proposition}

\newlength{\espaceavantspecialthm}
\newlength{\espaceapresspecialthm}
{\\}

\newenvironment{rema}[1][]{\refstepcounter{theorem} 
\noindent \textbf{Remark \thetheorem
#1.} }%

\newenvironment{defi*}[1][]{
\vskip \espaceavantspecialthm \noindent \textbf{D\'efinition.} }%
{\vskip \espaceapresspecialthm}

\begin{document}

\sloppy

\title{Conjugacy class of homeomorphisms and distortion elements in groups of homeomorphisms}
\author{E. Militon \thanks{Laboratoire J.A. Dieudonné. UMR n° 7351 CNRS UNS Université de Nice - Sophia Antipolis 06108 Nice Cedex 02 France e-mail: emmanuel.militon@unice.fr}}
\date{\today}
\maketitle

\setlength{\parskip}{10pt}

\selectlanguage{english}

\begin{abstract}
Let $S$ be a compact connected surface and let $f$ be an element of the group $\mathrm{Homeo}_{0}(S)$ of homeomorphisms of $S$ isotopic to the identity. Denote by $\tilde{f}$ a lift of $f$ to the universal cover of $S$. Fix a fundamental domain $D$ of this universal cover. The homeomorphism $f$ is said to be non-spreading if the sequence $(d_{n}/n)$ converges to $0$, where $d_{n}$ is the diameter of $\tilde{f}^{n}(D)$. Let us suppose now that the surface $S$ is orientable with a nonempty boundary. We prove that, if $S$ is different from the annulus and from the disc, a homeomorphism is non-spreading if and only if it has conjugates in $\mathrm{Homeo}_{0}(S)$ arbitrarily close to the identity. In the case where the surface $S$ is the annulus, we prove that a homeomorphism is non-spreading if and only if it has conjugates in $\mathrm{Homeo}_{0}(S)$ arbitrarily close to a rotation (this was already known in most cases by a theorem by Béguin, Crovisier, Le Roux and Patou). We deduce that, for such surfaces $S$, an element of $\mathrm{Homeo}_{0}(S)$ is distorted if and only if it is non-spreading.
\end{abstract}

\section{Conjugacy classes of non-spreading homeomorphisms}

The rotation number is a famous dynamical invariant introduced by Poincaré to study the dynamics of homeomorphisms of the circle. The dynamics of a homeomorphism of the circle will "look like" the dynamics of the rotation of angle $\alpha$ when the rotation number of this homeomorphism is $\alpha$. However, it is known that, for any $\alpha$, there exist homeomorphisms with rotation number $\alpha$ which are not conjugate to a rotation. One can solve this problem by classifying the homeomorphisms up to semi-conjugacy. There might be yet another approach to solve this problem: it is not difficult to prove the following proposition (see Section \ref{sectcircle} for more details).

\begin{proposition} \label{conjugacyclasscircle} For any homeomorphism of the circle with rotation number $\alpha$, the closure of the conjugacy class of this homeomorphism contains the rotation of angle $\alpha$. 
\end{proposition}

Actually, this last property characterizes the homeomorphisms of the circle with rotation number $\alpha$. In this article, we pursue this approach in the case of homeomorphisms of surfaces.

To generalize the notion of rotation number, Misiurewicz and Ziemian introduced the notion of rotation set of a homeomorphism of the two-dimensional torus isotopic to the identity (see \cite{MZ}). With the same approach, one can define the notion of rotation set of a homeomorphism of the closed annulus $\mathbb{A}=[0,1] \times \mathbb{S}^{1}$. Unlike the case of the circle, two orbits can have different asymptotic directions or different linear speeds: in those cases, the rotation set will contain more than one point and one can prove that the closure of the conjugacy class of the homeomorphism does not contain a rotation. Indeed the rotation set is continuous for the Hausdorff topology at rotations (see Corollary 3.7 in \cite{MZ}). Now, we investigate the case where the rotation set of the homeomorphism is reduced to a point. We call such homeomorphisms \emph{pseudo-rotations}. The only point in the rotation set of such a pseudo-rotation is called the \emph{angle} of this pseudo-rotation. In \cite{BCLRP}, Béguin, Crovisier, Le Roux and Patou proved the following theorem (see Corollary 1.2 in \cite{BCLRP}). The group $\mathrm{Homeo}_{0}(\mathbb{A})$ of homeomorphisms of $\mathbb{A}$ which are isotopic to the identity is endowed with the compact-open topology.

\begin{theorem}[Béguin-Crovisier-Le Roux-Patou] \label{BCLRP}
Let $f$ be a homeomorphism in $\mathrm{Homeo}_{0}(\mathbb{A})$. Suppose that $f$ is a pseudo-rotation of irrational angle $\alpha$. Then the closure of the conjugacy class of $f$ in $\mathrm{Homeo}_{0}(\mathbb{A})$ contains the rotation $R_{\alpha}$.
\end{theorem}

The following theorem is a consequence of Theorem 1.2 in \cite{Kwa}, which is due to Kwapisz.

\begin{theorem}[Kwapisz] \label{Kwapisz}
Let $f$ be a pseudo-rotation of $\mathbb{T}^{2}$ which is a $C^{1}$ diffeomorphism of $\mathbb{T}^2$. Suppose that there exists a representative $(\alpha_{1}, \alpha_{2})$ in $\mathbb{R}^{2}$ of the angle of the pseudo-rotation $f$ such that the real numbers $1$, $\alpha_{1}$ and $\alpha_{2}$ are $\mathbb{Q}$-linearly independent. Then the homeomorphism $f$ has conjugates in $\mathrm{Homeo}_{0}(\mathbb{T}^{2})$ arbitrarily close to the rotation of $\mathbb{T}^{2}$ defined by $(x,y) \mapsto (x+ \alpha_{1}, y + \alpha_{2})$.
\end{theorem}

 The above hypothesis on the angle $(\alpha_{1}, \alpha_{2})$ is the one which ensures that the rotation $(x,y) \mapsto (x+ \alpha_{1}, y + \alpha_{2})$ is minimal (\emph{i.e.} has no proper closed invariant set).

In this article, we investigate the case of rational pseudo-rotations of the annulus and homeomorphisms of compact surfaces $S$ with $\partial S \neq \emptyset$. We first introduce a more precise definition of pseudo-rotations of the annulus which will be useful later. Let $\mathbb{A}$ be the closed annulus $[0,1]\times \mathbb{S}^{1}$.

\begin{definition} \label{pseudo}
A homeomorphism $f$ in $\mathrm{Homeo}_{0}(\mathbb{A})$ is said to be a pseudo-rotation if there exists a lift $\tilde{f}: \mathbb{R} \times [0,1] \rightarrow \mathbb{R} \times [0,1]$ of the homeomorphism $f$ and a real number $\alpha$ such that
$$ \forall \tilde{x} \in \mathbb{R}\times [0,1], \ \lim_{n \rightarrow + \infty} \frac{p_{1}(\tilde{f}^{n}(\tilde{x}))}{n}= \alpha,$$
where $p_{1}: \mathbb{R} \times [0,1] \rightarrow \mathbb{R}$ is the projection. The class of $\alpha$ in $\mathbb{R}/ \mathbb{Z}$ is called the \emph{angle} of the pseudo-rotation $f$.
\end{definition}

Observe that the angle of a pseudo-rotation depends only on $f$ and not on the chosen lift $\tilde{f}$. 

We will prove in Section \ref{sectannulus} the following theorem.

\begin{theorem} \label{conjugacyclassannulus}
Let $f$ be a homeomorphism in $\mathrm{Homeo}_{0}(\mathbb{A})$. Suppose that $f$ is a pseudo-rotation of angle $\alpha$. Then the closure of the conjugacy class of $f$ in $\mathrm{Homeo}_{0}(\mathbb{A})$ contains the rotation $R_{\alpha}$, where
$$\begin{array}{rrcl}
R_{\alpha}: & \mathbb{A}= [0,1] \times \mathbb{R} / \mathbb{Z} & \rightarrow & \mathbb{A} \\
 & (t,x) & \mapsto & (t,x+ \alpha)
\end{array}
.$$ 
\end{theorem}

This theorem is an extension of Theorem \ref{BCLRP} in the case where the angle $\alpha$ is rational. We also have an analogous theorem in the case of the unit disc $\mathbb{D}^2$ of $\mathbb{R}^2$. In this case, for $\alpha \in \mathbb{R}/\mathbb{Z}$, if we see $\mathbb{D}^2$ as the complex unit disc, we define the rotation $R_{\alpha}$ as the map
$$\begin{array}{rrcl}
R_{\alpha}: & \mathbb{D}^2 & \rightarrow & \mathbb{D}^2 \\
 & z & \mapsto & e^{2i \pi \alpha}
\end{array}
.$$

\begin{theorem} \label{conjugacyclassdisc}
Let $f$ be homeomorphism in $\mathrm{Homeo}_{0}(\mathbb{D}^2)$. Suppose that its restriction to the boundary circle $\partial \mathbb{D}^2$ has rotation number $\alpha \in \mathbb{R}/ \mathbb{Z}$. Then the homeomorphism $f$ has conjugates arbitrarily close to the rotation $R_{\alpha}$.
\end{theorem}

To state the next theorem, we need to extend the notion of pseudo-rotation to the context of a homeomorphism of an arbitrary surface. Let $S$ be a surface. We denote by $\tilde{S}$ its universal cover which we endow with a "natural" distance, \emph{i.e.} a distance which is invariant under the group of deck transformations. For a subset $A \subset \tilde{S}$, we denote by $\mathring{A}$ its interior and by $\mathrm{diam}(A)$ its diameter.

\begin{definition}
We call \emph{fundamental domain} of $\tilde{S}$ (for the action of the group $\pi_{1}(S)$ of deck transformations of $\tilde{S}$) any compact connected subset $D$ of $\tilde{S}$ which satisfies the following properties:
\begin{enumerate}
\item $\Pi(D)=S$, where $\Pi: \tilde{S} \rightarrow S$ is the projection.
\item For any deck transformation $\gamma$ in $\pi_{1}(S)$ different from the identity, $\mathring{D} \cap \mathring{\gamma(D)}= \emptyset$. 
\end{enumerate}
\end{definition}

Fix a fundamental domain $D$ for the action of the group of deck transformations of the covering $\tilde{S} \rightarrow S$. For any homeomorphism $f$ of $S$ isotopic to the identity, we denote by $\tilde{f}: \tilde{S} \rightarrow \tilde{S}$ a lift of $f$ which is the time one of the lift starting from the identity of an isotopy between the identity and the homeomorphism $f$. By classical results by Hamström (see \cite{Ham}), such a homeomorphism $\tilde{f}$ is unique if the surface is different from the torus, the Klein bottle, the Möbius strip or the annulus. Moreover, for any deck transformation $\gamma$, $\gamma \tilde{f}= \tilde{f} \gamma$, by uniqueness of the lift of an isotopy from the identity to $f$ starting from $\gamma$. Denote by $\mathrm{Homeo}_{0}(S)$ the group of homeomorphisms of $S$ isotopic to the identity.

\begin{definition}
A homeomorphism $f$ in $\mathrm{Homeo}_{0}(S)$ is called non-spreading if $\lim_{n \rightarrow + \infty} \frac{\mathrm{diam}(\tilde{f}^{n}(D))}{n}=0$. 
\end{definition}

\begin{rema}
The condition $\lim_{n \rightarrow + \infty} \frac{\mathrm{diam}(\tilde{f}^{n}(D))}{n}=0$ is independent of the chosen fundamental domain $D$ (see Proposition 3.4 in \cite{Mil}).
\end{rema}

\begin{rema} In the case where the surface is an annulus or the torus, a homeomorphism is non-spreading if and only if it is a pseudo-rotation. Indeed, in the case of the torus, the sequence $\frac{1}{n} \tilde{f}^{n}(D)$ of compact subsets of $\mathbb{R}^{2}$ converges to the rotation set of $\tilde{f}$ for the Hausdorff topology  (see \cite{MZ}).
\end{rema}

In Section 5, we prove the following theorem.

\begin{theorem} \label{conjugacyclassboundary}
Let $S$ be a compact surface with $\partial S \neq \emptyset$ which is different from the disc, the annulus or the Möbius strip. For any non-spreading homeomorphism $f$ of $S$, the closure of the conjugacy class of $f$ in $\mathrm{Homeo}_{0}(S)$ contains the identity.
\end{theorem}

\begin{rema}
The theorem remains true when we replace the group $\mathrm{Homeo}_{0}(S)$ with the identity component $\mathrm{Homeo}_{0}(S, \partial S)$ of the group of homeomorphisms of $S$ which pointwise fix a neighbourhood of the boundary. The proof in this case is almost identical to the proof in the case of the group $\mathrm{Homeo}_{0}(S)$.
\end{rema}

\begin{rema} This property characterizes the non-spreading homeomorphisms of such a surface $S$: if a homeomorphism isotopic to the identity satisfies this property, then we will see that it is a distorted element in $\mathrm{Homeo}_{0}(S)$ (the notion of distorted elements will be explained in the next section). Moreover, we will see in the next section that distortion elements in $\mathrm{Homeo}_{0}(S)$ are non-spreading.
\end{rema}

The following conjecture is natural.

\begin{conjecture} \label{conjugacyclasstorus} The closure of the conjugacy class of any pseudo-rotation of the torus of angle $\alpha$ contains the translation of angle $\alpha$. The closure of the conjugacy class of any non-spreading homeomorphism of a closed surface $S$ of genus $g \geq 2$ contains the identity.
\end{conjecture}

\section{Distortion elements in groups of homeomorphisms of surfaces}

In this article, we pursue the study of distorted elements (see definition below) in groups of homeomorphisms of manifolds initiated in the article \cite{Mil}. For more background on distortion elements in groups of homeomorphisms or diffeomorphisms of manifolds, see \cite{Mil}.

Let $G$ be a finitely generated group and $\mathcal{G}$ be a finite generating set of $G$. For any element $g$ in $G$, we denote by $l_{\mathcal{G}}(g)$ the minimal number of factors in a decomposition of $g$ as a product of elements of $\mathcal{G} \cup \mathcal{G}^{-1}$.

\begin{definition}
Let $G$ be any group. An element $g$ in $G$ is said to be \emph{distorted} (or is a \emph{distortion element}) if there exists a finite subset $\mathcal{G} \subset G$ such that the following properties are satisfied.
\begin{enumerate}
\item The element $g$ belongs to the group generated by $\mathcal{G}$.
\item $\lim_{n \rightarrow + \infty} l_{\mathcal{G}}(g^{n})/n=0$.
\end{enumerate}
\end{definition}

Let $M$ be a compact manifold. We denote by $\mathrm{Homeo}_{0}(M)$ the group of compactly-supported homeomorphisms which are isotopic to the identity. We endow the group $\mathrm{Homeo}_{0}(M)$ with the compact-open topology. For any manifold $M$, we denote by $\tilde{M}$ its universal cover.

Recall the following easy proposition (see Proposition 2.4 in \cite{Mil} for a proof).

\begin{proposition} \label{distorsionimpliquecroissance}
Let $S$ be a compact surface. Denote by $D$ a fundamental domain of $\tilde{S}$ for the action of $\pi_{1}(S)$.\\
If a homeomorphism $f$ in $\mathrm{Homeo}_{0}(S)$ is a distortion element in $\mathrm{Homeo}_{0}(S)$, then $f$ is non-spreading.
\end{proposition}

We conjectured that an element in $\mathrm{Homeo}_{0}(S)$ which satisfies the conclusion of this proposition is distorted. However, we were not able to prove it and we just proved the following weaker statement (see Theorem 2.6 in \cite{Mil}).

\begin{theorem} \label{croissanceimpliquedistorsion}
Let $f$ be a homeomorphism in $\mathrm{Homeo}_{0}(S)$. If
$$\liminf_{n \rightarrow + \infty} \frac{\mathrm{diam}(\tilde{f}^{n}(D)) \mathrm{log}(\mathrm{diam}(\tilde{f}^{n}(D)))}{n}=0,$$
then $f$ is a distortion element in $\mathrm{Homeo}_{0}(S)$.
\end{theorem}

In this article, we try to improve the above result. We will prove the following theorem which is the key idea to obtain this improvement.

\begin{definition}
Let $f$ be a homeomorphism in $\mathrm{Homeo}_{0}(M)$. A \emph{conjugate of $f$} is a homeomorphism of the form $hfh^{-1}$, where $h$ is any element of $\mathrm{Homeo}_{0}(M)$.
\end{definition}

\begin{theorem} \label{distcont}
Let $f \in \mathrm{Homeo}_{0}(M)$. Suppose that the homeomorphism $f$ has conjugates arbitrarily close to an element of $\mathrm{Homeo}_{0}(M)$ which is distorted. Then the element $f$ is distorted in $\mathrm{Homeo}_{0}(M)$.
\end{theorem}

To prove this theorem, we will find a map $\mathrm{Homeo}_{0}(M) \rightarrow \mathbb{R}$ which vanishes exactly on the distortion elements of $\mathrm{Homeo}_{0}(M)$ and which is continuous at those distortion elements. In the case of the 2-dimensional torus, notice that, if Conjecture \ref{conjugacyclasstorus} and Theorem \ref{distcont} are true, the map $f \mapsto \lim_{n \rightarrow +\infty} \frac{\mathrm{diam}(\tilde{f}^{n}(D))}{n}$ satisfies those two conditions. Indeed, the limit $\lim_{n \rightarrow +\infty} \frac{\mathrm{diam}(\tilde{f}^{n}(D))}{n}$, which always exists and is finite, is the diameter of the rotation set of $f$ which is known to be upper semi-continuous (see Corollary 3.7 in \cite{MZ}). In \cite{Mil}, we define a useful quantity which vanishes exactly on the distortion elements of $\mathrm{Homeo}_{0}(M)$ (see Proposition 4.1 in \cite{Mil}). However, we do not know whether this quantity is continuous at the distortion elements of $\mathrm{Homeo}_{0}(M)$. That is why we slightly changed the definition of this quantity to obtain maps $g_{C}$ and $G_{C}$ which vanish exactly on the distortion elements of $\mathrm{Homeo}_{0}(M)$ and which are continuous at those distortion elements. 

We are interested now in the case where the manifold $M$ is a surface. Let $S$ be a compact surface with $\partial S \neq \emptyset$ which is different from the Möbius strip. Using Theorems \ref{conjugacyclassannulus}, \ref{conjugacyclassdisc}, \ref{conjugacyclassboundary}, Proposition \ref{distorsionimpliquecroissance} and Theorem \ref{distcont}, we obtain a complete dynamical description of the distortion elements of the group $\mathrm{Homeo}_{0}(S)$.

\begin{corollary} \label{caractdistann}
An element $f$ in $\mathrm{Homeo}_{0}(S)$ is distorted if and only if it is a non-spreading homeomorphism.
\end{corollary}

\begin{proof}
The "only if" implication is a consequence of Proposition \ref{distorsionimpliquecroissance}. The "if" implication involves Theorems \ref{distcont}, \ref{conjugacyclassannulus}, \ref{conjugacyclassboundary} and the fact that a rotation of the annulus is a distortion element in $\mathrm{Homeo}_{0}(\mathbb{A})$. This last fact is a straightforward consequence of Theorem \ref{croissanceimpliquedistorsion}. In the case of the disc, the "if" implication is a consequence of Theorem \ref{croissanceimpliquedistorsion}.
\end{proof}

Note that a rational pseudo-rotation of the annulus has a power which is a pseudo-rotation of angle $0$. Moreover, if an element of a group admits a positive power which is distorted, then this element is distorted. Hence the case $\alpha=0$ in Theorem \ref{conjugacyclassannulus} (together with Theorem \ref{BCLRP} by Béguin, Crovisier, Le Roux and Patou) is sufficient to obtain Corollary \ref{caractdistann} in the case of the annulus.

Using Theorems \ref{distcont} and \ref{Kwapisz}, we obtain the following corollary.

\begin{corollary}
Let $f$ be a pseudo-rotation of $\mathbb{T}^{2}$ which is a $C^1$ diffeomorphism of $\mathbb{T}^2$. Suppose that there exists a representative $(\alpha_{1}, \alpha_{2})$ in $\mathbb{R}^{2}$ of the angle of the homeomorphism $f$ such that the real numbers $1$, $\alpha_{1}$ and $\alpha_{2}$ are $\mathbb{Q}$-linearly independent. Then the element $f$ is distorted in the group $\mathrm{Homeo}_{0}(\mathbb{T}^{2})$.
\end{corollary}

\section{Stability properties of distortion elements}

In this section, we prove Theorem \ref{distcont}. Let $\overline{B}(0,1)$ be the unit closed ball in $\mathbb{R}^{d}$ and 
$$H^{d}=\left\{ (x_{1},x_{2}, \ldots, x_{d}) \in \mathbb{R}^{d}, x_{1} \geq 0 \right\}.$$

\begin{definition} A subset $B$ of $M$ is called a \emph{closed ball} if there exists an embedding $e:\mathbb{R}^{d} \rightarrow M$ such that $e(\overline{B}(0,1))=B$.
We call \emph{closed half-ball} of $M$ the image of the \emph{unit half-ball} $\overline{B}(0,1) \cap H^{d}$ under an embedding $e:H^{d} \rightarrow M$ such that
$$e(\partial H^{d})=e(H^{d}) \cap \partial M.$$
\end{definition}

Let us fix a finite family $\mathcal{U}$ of closed balls or closed half-balls whose interiors cover $M$. We denote by $N(\mathcal{U})$ the cardinality of this cover. We need the following lemma which is proved in Section \ref{defa}.

\begin{lemma} \label{abiendefini}
Let $f$ be a homeomorphism in $\mathrm{Homeo}_{0}(M)$. Then there exists a finite family $(f_{i})_{1 \leq i \leq s}$ of homeomorphisms in $\mathrm{Homeo}_{0}(M)$ such that the following properties are satisfied:
\begin{enumerate}
\item Each homeomorphism $f_{i}$ is supported in the interior of one of the sets in $\mathcal{U}$. 
\item $f=f_{1} \circ f_{2} \circ \ldots \circ f_{s}$.
\item The cardinality of the set $\left\{ f_{i}, 1 \leq i \leq s \right\}$ is less than or equal to $5 N(\mathcal{U})$.
\end{enumerate}
\end{lemma}

Let $C$ be an integer which is greater than or equal to $5N(\mathcal{U})$. Let $f$ be any homeomorphism in $\mathrm{Homeo}_{0}(M)$. We denote by $a_{C}(f)$ the minimal integer $s$ such that the following property is satisfied. There exists a finite family $(f_{i})_{1 \leq i \leq s}$ of homeomorphisms in $\mathrm{Homeo}_{0}(M)$ such that:
\begin{enumerate}
\item Each homeomorphism $f_{i}$ is supported in the interior of one of the sets in $\mathcal{U}$. 
\item $f=f_{1} \circ f_{2} \circ \ldots \circ f_{s}$.
\item The cardinality of the set $\left\{ f_{i}, 1 \leq i \leq s \right\}$ is less than or equal to $C$.
\end{enumerate}
Let $g_{C}(f)= \liminf_{ n \rightarrow + \infty} a_{C}(f^{n})/n$ and $G_{C}(f)= \limsup_{n \rightarrow + \infty} a_{C}(f^{n})/n$.

In order to prove Theorem \ref{distcont}, we need the following results which will be proved afterwards. 

This first lemma says that, essentially, the quantities $g_{C}$ and $G_{C}$ are the same and do not really depend on $C$.

\begin{lemma} \label{gandG}
Let $C>C'\geq 5 N(\mathcal{U})$ be integers. The following properties hold.
\begin{enumerate}
\item $G_{C+5N(\mathcal{U})} \leq g_{C} \leq G_{C} < + \infty$.
\item $a_{C} \leq a_{C'}$.
\item $a_{5N(\mathcal{U})} \leq (14 \mathrm{log}(C)+14) a_{C}$.
\end{enumerate}
\end{lemma}

Hence, if $C>C'\geq 5 N(\mathcal{U})$, then $g_{C} \leq g_{C'}$, $G_{C} \leq G_{C'}$, $g_{5N(\mathcal{U})} \leq (14 \mathrm{log}(C)+14)g_{C}$  and $G_{5N(\mathcal{U})} \leq (14 \mathrm{log}(C)+14)G_{C}$.

This lemma is easy to prove once we have Lemma \ref{Avila} below in mind. It is proved in Subsection \ref{propgCGC}.

The two following propositions are the two main steps of the proof of Theorem \ref{distcont}. The first one says that the quantities $g_{C}$ (or equivalently $G_{C}$, $g_{5N(\mathcal{U})}$ or $G_{5N(\mathcal{U})}$) vanish exactly on the distortion elements of $\mathrm{Homeo}_{0}(M)$. The second one is a continuity property of those quantities.

\begin{proposition} \label{caractdist}
Let $f$ be a homeomorphism in $\mathrm{Homeo}_{0}(M)$. The following conditions are equivalent:
\begin{enumerate}
\item The element $f$ is distorted in the group $\mathrm{Homeo}_{0}(M)$.
\item There exists an integer $C \geq 5 N(\mathcal{U})$ such that $G_{C}(f)=0$.
\item There exists an integer $C \geq 5 N(\mathcal{U})$ such that $g_{C}(f)=0$.
\item $G_{5N(\mathcal{U})}(f)=0$.
\item $g_{5N(\mathcal{U})}(f)=0$.
\end{enumerate}
\end{proposition}
The equivalence between the the four last assertions follows from Lemma \ref{gandG}. The equivalence with the first item, which is proved at the end of this section, is a consequence of a deeper result in \cite{Mil}.

The following proposition is proved in Subsection 3.2.

\begin{proposition} \label{continuity}
Let $C \geq 5N(\mathcal{U})$ be an integer. The map $g_{C}: \mathrm{Homeo}_{0}(M) \rightarrow \mathbb{R}$ is continuous at the distortion elements of the group $\mathrm{Homeo}_{0}(M)$.
\end{proposition}

Observe that, by Lemma \ref{gandG} and Proposition \ref{caractdist}, this proposition implies that the maps $G_{ C}$ are also continuous at the distortion elements of the group $\mathrm{Homeo}_{0}(M)$. Before proving the above propositons, let us prove Theorem \ref{distcont}.

\begin{proof}[Proof of Theorem \ref{distcont}]
By Proposition \ref{caractdist}, it suffices to prove that $g_{10 N(\mathcal{U})}(f)=0$. Denote by $g$ a distortion element in $\mathrm{Homeo}_{0}(M)$ which belongs to the closure of the set of conjugates of $f$. Observe first that, for any homeomorphism $h$ in $\mathrm{Homeo}_{0}(M)$ and any integer $n$, 
$$ a_{10N(\mathcal{U})}(f^{n}) \leq a_{5N(\mathcal{U})}(hf^{n}h^{-1})+2a_{5N(\mathcal{U})}(h).$$
Hence
$$g_{10N(\mathcal{U})}(f) \leq g_{5N(\mathcal{U})}(hfh^{-1}).$$
 Recall that, by Proposition \ref{caractdist}, $g_{10N(\mathcal{U})}(g)=0$. By Proposition \ref{continuity}, the right-hand side of the last inequality can be chosen to be arbitrarily small. Therefore $g_{10 N(\mathcal{U})}(f)=0$. Proposition \ref{caractdist} implies that $f$ is distorted in the group $\mathrm{Homeo}_{0}(M)$.
\end{proof}

The above Lemmas will be essentially consequences of the following Lemma which is proved in \cite{Mil} (see Lemma 4.5 and its proof).

\begin{lemma} \label{Avila}
Let $(f_{n})_{n \in \mathbb{N}}$ be a sequence of homeomorphisms of $\mathbb{R}^{d}$ (respectively of $H^{d}$) supported in the unit ball (respectively the unit half-ball). Then there exists a finite set $\mathcal{G} \subset \mathrm{Homeo}_{c}(\mathbb{R}^{d})$ (respectively $\mathcal{G} \subset \mathrm{Homeo}_{c}(H^{d})$) such that:
\begin{enumerate}
\item $\sharp \mathcal{G} \leq 5$.
\item For any integer $n$, the element $f_{n}$ belongs to the group generated by $\mathcal{G}$.
\item For any integer $n$, $l_{\mathcal{G}}(f_{n}) \leq 14 \log(n)+14$.
\end{enumerate}
\end{lemma}

\subsection{The quantity $a_{5 N(\mathcal{U})}$ is well defined}
\label{defa}

In this section, we prove Lemma \ref{abiendefini}.
\begin{proof}
Take a homeomorphism $f$ in $\mathrm{Homeo}_{0}(M)$. We now apply the following classical result, called the fragmentation lemma.

\begin{lemma} \label{fragmentationlemma}
Let $f$ be a homeomorphism in $\mathrm{Homeo}_{0}(M)$. Then there exist an integer $k \geq 0$ and homeomorphisms $f_{1}, f_{2}, \ldots, f_{k}$ in $\mathrm{Homeo}_{0}(M)$ such that:
\begin{enumerate}
\item Each homeomorphism $f_{i}$ is supported in the interior of one of the sets of $\mathcal{U}$.
\item $f=f_{1} \circ f_{2} \circ \ldots \circ f_{k}$.
\end{enumerate}
Moreover, there exist a constant $C(\mathcal{U})>0$ and a neighbourhood of the identity such that any homeomorphism $f$ in this neighbourhood admits a decomposition as above with $k \leq C(\mathcal{U})$.
\end{lemma}

A proof of this Lemma can be found in \cite{Bou} or \cite{Fis}, for instance.  The idea is to prove the lemma for homeomorphisms sufficiently close to the identity and then use a connectedness argument to extend this result to any homeomorphism in $\mathrm{Homeo}_{0}(M)$.

The fragmentation lemma applied to our homeomorphism $f$ yields a decomposition $f=f_{1} \circ f_{2} \circ \ldots \circ f_{k}$. Consider now a partition $\left\{A_{U}, U \in \mathcal{U} \right\}$ of the set $\left\{1, 2, \ldots, k \right\}$ such that, for any set $U$ in $\mathcal{U}$ and any index $i$ in $A_{U}$, the homeomorphism $f_{i}$ is supported in the interior of $U$. Now, for each element $U$ of our cover $\mathcal{U}$, we apply Lemma \ref{Avila} to the finite sequence $(f_{i})_{i \in A_{U}}$: this provides a decomposition of each of the $f_{i}$'s. The concatenation of those decompositions gives a decomposition of our homeomorphism $f$ which satisfies the conclusion of Lemma \ref{abiendefini}.
\end{proof}

\subsection{Properties of the maps $g_{C}$ and $G_{C}$} \label{propgCGC}

In this subsection, we prove Lemma \ref{gandG}, Proposition \ref{caractdist} and Proposition \ref{continuity}. These results rely on the following facts.

Let $C, C' \geq 5 N(\mathcal{U})$, $p>0$ and $f$ and $g$ be elements of $\mathrm{Homeo}_{0}(M)$.\\
\underline{Fact 1}: $a_{C+C'}(fg) \leq a_{C}(f)+a_{C'}(g)$.\\
\underline{Fact 2}: $a_{C}(f^{p}) \leq p a_{C}(f)$.

\begin{proof}[Proof of Lemma \ref{gandG}]
The inequalities $g_{C} \leq G_{C}$ and $a_{C} \leq a_{C'}$ are obvious. Fix an integer $k>0$. Take any integer $n>0$ and perform the Euclidean division: $n=qk+r$. By Facts 1 and 2,
$$\begin{array}{rcl}
a_{C+5N(\mathcal{U})}(f^{n}) & \leq & a_{C}(f^{qk})+ a_{5N(\mathcal{U})}(f^{r})\\
 & \leq & q a_{C}(f^{k})+ a_{5N(\mathcal{U})}(f^{r}).
\end{array}
$$
Hence, dividing by $n$, and taking the upper limit as $n \rightarrow + \infty$, we obtain that $G_{C+ 5 N(\mathcal{U})}(f) \leq a_{C}(f^{k})/k$. This relation implies the inequality $G_{C+5N(\mathcal{U})} \leq g_{C}$. It implies moreover that, for any $C \geq 5 N(\mathcal{U})$, $G_{C+5N(\mathcal{U})} < + \infty$. From the inequality $a_{5N(\mathcal{U})} \leq (14 \mathrm{log}(C)+14)a_{C}$, which we prove below, we deduce that $G_{5N(\mathcal{U})} < + \infty$. Hence, for any $C \geq 5N( \mathcal{U})$, $G_{C} \leq G_{5N(\mathcal{U})} < + \infty$ and $g_{C} \leq g_{5N(\mathcal{U})} \leq G_{5 N(\mathcal{U})} < + \infty$.

It remains to prove that $a_{5 N(\mathcal{U})} \leq (14 \mathrm{log}(C)+14)a_{C}$. Let $f$ be an element of $\mathrm{Homeo}_{0}(M)$ and let $l=a_{C}(f)$. By definition, there exists a map $\sigma: \left\{1, \ldots, l\right\} \rightarrow \left\{ 1, \ldots, C \right\}$ and elements $f_{1}, \ldots, f_{C}$ of $\mathrm{Homeo}_{0}(M)$ supported in the interior of one of the sets of $\mathcal{U}$ such that:
$$ f= f_{\sigma(1)}  f_{\sigma(2)} \ldots f_{\sigma(l)}.$$
Denote by $\left\{ A_{U}, U \in \mathcal{U} \right\}$ a partition of the set $\left\{ 1, \ldots, C \right\}$ such that, for any set $U$ in $\mathcal{U}$:
$$ \forall j \in A_{U}, \mathrm{supp}(f_{j}) \subset \mathring{U}.$$
Fix such an open set $U$. Lemma \ref{Avila} applied to the (finite) sequence $(f_{j})_{j \in A_{U}}$ provides a finite set $\mathcal{G}_{U}$ of homeomorphisms which are each supported in $\mathring{U}$ such that the following properties are satisfied.
\begin{enumerate}
\item For any index $j$ in $A_{U}$, the element $f_{j}$ belongs to the group generated by $\mathcal{G}_{U}$.
\item The set $\mathcal{G}_{U}$ contains at most $5$ elements.
\item $l_{\mathcal{G}_{U}}(f_{j}) \leq 14 \log(\sharp A_{U})+14.$
\end{enumerate}
Hence, if we take $\mathcal{G}= \bigcup \limits _{U \in \mathcal{U}} \mathcal{G}_{U}$, we have:
\begin{enumerate}
\item $\sharp \mathcal{G} \leq 5 N(\mathcal{U})$.
\item The element $f$ belongs to the group generated by $\mathcal{G}$.
\item $l_{\mathcal{G}}(f) \leq (14 \log(C)+14)l$.
\end{enumerate}
Hence $a_{5N(\mathcal{U})}(f) \leq (14 \log(C)+14) a_{ C}(f)$.
\end{proof}

\begin{proof}[Proof of Proposition \ref{caractdist}]
The equivalence between the last four conditions is a direct consequence of Lemma \ref{gandG}. By Lemma \ref{fragmentationlemma}, if the element $f$ is distorted, then there exists $C$ such that $G_{C}(f)=0$: it suffices to apply this lemma to each factor provided by the definition of a distortion element. Conversely, if there exists $C$ such that $G_{C}(f)=0$, then $f$ is distorted by Proposition 4.1 in \cite{Mil}.  
\end{proof}

\begin{proof}[Proof of Proposition \ref{continuity}] 
By Lemma \ref{gandG} and Proposition \ref{caractdist}, it suffices to prove the lemma for $C= 15N(\mathcal{U})$. Fix $\epsilon>0$. Let $f$ be an element which is distorted in $\mathrm{Homeo}_{0}(M)$. By Proposition \ref{caractdist}, we can find an integer $p>0$ such that 
$$\frac{a_{5N(\mathcal{U})}(f^{p})}{p} + \frac{(14 \mathrm{log}(C(\mathcal{U}))+14)C(\mathcal{U})}{p} < \epsilon,$$ where $C(\mathcal{U})$ is given by Lemma \ref{fragmentationlemma}. Take a homeomorphism $g$ in $\mathrm{Homeo}_{0}(M)$ sufficiently close to $f$ so that $h=f^{-p}g^{p}$ belongs to the neighbourhood given by Lemma \ref{fragmentationlemma}.  For any positive integer $n$, we write $n=pq_{n}+r_{n}$, where $q_{n}$ and $r_{n}$ are respectively the quotient and the remainder of the Euclidean division of $n$ by $p$. By Facts 1 and 2, 
$$a_{15N(\mathcal{U})}(g^{n}) \leq q_{n}a_{10N(\mathcal{U})}(g^{p})+ a_{5N(\mathcal{U})}(g^{r_{n}}).$$
Dividing by $n$ and taking the lower limit when $n$ tends to $+ \infty$, 
$$ g_{15N(\mathcal{U})}(g) \leq a_{10N(\mathcal{U})}(g^{p})/p.$$
By Fact 1,
$$a_{10N(\mathcal{U})}(g^{p}) \leq a_{5N(\mathcal{U})}(f^{p})+ a_{5N(\mathcal{U})}(h).$$
By Lemma \ref{fragmentationlemma}, $a_{\max(C(\mathcal{U}), 5 N( \mathcal{U}))}(h) \leq C(\mathcal{U})$ and, by Lemma \ref{gandG}, $a_{5N(\mathcal{U})}(h) \leq (14 \mathrm{log}(C(\mathcal{U}))+14)a_{\max(C(\mathcal{U}),5N(\mathcal{U}))}(h)$ . Hence
$$ g_{15N(\mathcal{U})}(g) \leq \frac{a_{5N(\mathcal{U})}(f^{p})}{p} + \frac{(14 \mathrm{log}(C(\mathcal{U}))+14)C(\mathcal{U})}{p} < \epsilon.$$
\end{proof}

\section{Conjugacy classes: case of the circle} \label{sectcircle}

In this section, we prove Proposition \ref{conjugacyclasscircle}. Denote by $\mathrm{Homeo}_{0}(\mathbb{S}^{1})$ the group of orientation-preserving homeomorphisms of the circle $\mathbb{S}^{1}= \mathbb{R} / \mathbb{Z}$. In this section, for any $\alpha \in \mathbb{R}/ \mathbb{Z}$, we denote by $R_{\alpha}$ the rotation of the circle $x \mapsto x+ \alpha$. If $f$ denotes an orientation-preserving homeomorphism of the circle, we denote by $\rho(f)$ its rotation number.

Proposition \ref{conjugacyclasscircle} is not difficult to prove and one might find more straightforward proofs of it using semi-conjugacy results. However, we will use the proof given here for the case of homeomorphisms of the disc. Hence this section can be considered as a preparatory section for the case of the disc.

We will make a distinction between the case of a homeomorphism with a irrational rotation number and the case of a homeomorphism with an rational rotation number. Let us start with the irrational case.

\begin{proposition} \label{finiteconj}
Fix a homeomorphism $f$ in $\mathrm{Homeo}_{0}(\mathbb{S}^{1})$, a point $x$ in $\mathbb{S}^{1}$ and an integer $N \geq 0$. Suppose that the rotation number of $f$ is irrational. Then there exists a homeomorphism $h$ in $\mathrm{Homeo}_{0}(\mathbb{S}^{1})$ such that, for any $0 \leq k \leq N$,
$$h(f^{k}(x))=R_{\alpha}^{k}(x).$$ 
\end{proposition}

Of course this proposition does not hold in the case that the rotation number of $f$ is rational as $f$ can have infinite orbits whereas any orbit under a rational rotation is finite.

\begin{proof}
Denote by $x$, $R_{\alpha}^{k_{1}}(x)$,..., $R_{\alpha}^{k_{N}}(x)$ the points of $\left\{ R_{\alpha}^{k}(x), 0 \leq k \leq N \right\}$ which are successively met when we follow the circle in the sense given by the orientation of the circle, starting from the point $x$. Then, by Proposition 11.2.4 P.395 in \cite{KH}, the points which we meet successively among the points $f^{k}(x)$, for $0 \leq k \leq N$, when we follow the oriented circle starting from $x$, are $x$, $f^{k_{1}}(x)$,..., $f^{k_{N}}(x)$. Hence there exists a homeomorphism $h$ in $\mathrm{Homeo}_{0}(\mathbb{S}^{1})$ which, for any $0 \leq i \leq N$, sends the interval $[f^{k_{i}}(x),f^{k_{i+1}}(x)]$ onto the interval $[R_{\alpha}^{k_{i}}(x),R_{\alpha}^{k_{i+1}}(x)]$., where $k_{0}=0$ and $k_{N+1}=0$. The homeomorphism $h$ satisfies the required property.
\end{proof}

\begin{corollary}
Let $f$ be an orientation-preserving homeomorphism of the circle with $\rho(f)=\alpha$ irrational. Then the homeomorphism $f$ has conjugates in $\mathrm{Homeo}_{0}(\mathbb{S}^{1})$ arbitrarily close to the rotation $R_{\alpha}$.
\end{corollary}

\begin{proof}
Let $\epsilon >0$. Fix a point $x_{0}$ of the circle. As the orbits under the rotation $R_{\alpha}$ are dense in the circle, one can find $N>0$ such that the length of any connected component of the complement of $\left\{ R_{\alpha}^{k}(x_{0}), 0 \leq k \leq N-1 \right\}$ is smaller than $\epsilon$. Proposition \ref{finiteconj} yields a homeomorphism $h$ such that, for any $0 \leq k \leq N$, $h(f^{k}(x_{0}))=R_{\alpha}^{k}(x_{0})$. As the point $x_{0}$ is fixed under $h^{-1}$, we also have, for any $0 \leq k \leq N$, $hf^{k}h^{-1}(x_{0})=R_{\alpha}^{k}(x_{0})$.

Denote by $(R_{\alpha}^{k_{0}}(x_{0}), R_{\alpha}^{k_{1}}(x_{0}))$ any connected component of the complement of $\left\{ R_{\alpha}^{k}(x_{0}), 0 \leq k \leq N-1 \right\}$ in the circle. For any point $x$ in $[R_{\alpha}^{k_{0}}(x_{0}), R_{\alpha}^{k_{1}}(x_{0})]$, the point $hfh^{-1}(x)$ belongs to the interval
$$[hfh^{-1}R_{\alpha}^{k_{0}}(x_{0}),hfh^{-1}R_{\alpha}^{k_{1}}(x_{0})]=[R_{\alpha}^{k_{0}+1}(x_{0}), R_{\alpha}^{k_{1}+1}(x_{0})].$$
As the point $R_{\alpha}(x)$ also belongs to this interval and as the length of this interval is smaller than $\epsilon$,
$$d(hfh^{-1}(x),R_{\alpha}(x)) <\epsilon.$$
\end{proof}

Now, we deal with the case where the rotation number is rational.

\begin{proposition} \label{finiteconjrat}
Let $f$ be a homeomorphism in $\mathrm{Homeo}_{0}(\mathbb{S}^{1})$. Suppose that $\rho(f)= \frac{p}{q}$, where $p$ and $q$ are relatively prime integers. Fix a (large) integer $N>0$. There exists a cover $(I_{j})_{j \in \mathbb{Z}/Nq\mathbb{Z}}=([a_{j},b_{j}])_{j \in \mathbb{Z}/Nq\mathbb{Z}}$ of the circle by intervals whose interiors are pairwise disjoint with the following properties.
\begin{enumerate}
\item for any $j$, $a_{j+1}=b_{j}$.
\item for any $j$, $f(I_{j}) \subset I_{j+Np-1}\cup I_{j+Np} \cup I_{j+Np+1}$.
\end{enumerate}
\end{proposition}

\begin{proof}
By a classical result by Poincaré (see Proposition 11.1.4 in \cite{KH}), there exists a point $x_{0}$ of the circle which is periodic for $f$ with period $q$, \emph{i.e.} $f^{q}(x_{0})=x_{0}$ and $f^{k}(x_{0}) \neq x_{0}$ whenever $0<k<q$. Denote by $x_{0}=f^{k_{0}}(x_{0})$, $f^{k_{1}}(x_{0})$,..., $f^{k_{q-1}}(x_{0})$ the points of $\left\{ f^{k}(q), k \in \mathbb{Z}/q \mathbb{Z} \right\}$ in the order given by the orientation of the circle. Construct by induction $N-1$ points $x_{1},x_{2}, \ldots, x_{N-1}$ in the open interval $(x_{0},f^{k_{1}}(x_{0}))$ such that
$$x_{0} <x_{1} <x_{2} < \ldots <x_{N-1}<f^{k_{1}}(x_{0})$$
and, for any $0<i<N-1$, $x_{i-1}<f^{q}(x_{i})<x_{i+1}$.

It suffices to take the connected components of the complement of 
$$\left\{f^{k}(x_{j}), \left\{ 
\begin{array}{l}
0 \leq k \leq q-1 \\ 
0 \leq j \leq N-1
\end{array}
\right. \right\}$$
as intervals $I_{j}$. More precisely, for any $0 \leq j \leq N-2$ and any $0 \leq i \leq q-1$, take
$$I_{j+Ni}=f^{k_{i}}([x_{j},x_{j+1}]),$$ 
and, for any $0 \leq i \leq q-1$,
$$I_{N-1+Ni}=[f^{k_{i}}(x_{N-1}),f^{k_{i+1}}(x_{0})].$$
As the points $f^{k}(x_{0})$ are in the same order on the circle as the points $R_{\frac{p}{q}}^{k}(x_{0})$, we obtain that, for any index $i \in \mathbb{Z}/q \mathbb{Z}$, $f(f^{k_{i}}(x_{0}))=f^{k_{i+p}}(x_{0})$. Hence, if $k_{i} \neq q-1$, $f(I_{j})=I_{j+Np}$ and, if $k_{j}=q-1$, $f(I_{j}) \subset I_{j+Np-1} \cup I_{j+Np} \cup I_{j+Np+1}$ as, for any $i$, $x_{i-1}<f^{q}(x_{i})<x_{i+1}$.
\end{proof}

\begin{corollary}
Let $f$ be a homeomorphism in $\mathrm{Homeo}_{0}(\mathbb{S}^{1})$. Suppose that $\rho(f)$ is rational. Then the homeomorphism $f$ has conjugates arbitrarily close to the rotation $R_{\frac{p}{q}}$.
\end{corollary}

\begin{proof}
Let $\epsilon >0$ and take an integer $N$ sufficiently large such that $\frac{1}{Nq} < \frac{\epsilon}{3}$. Set $I'_{j}=[\frac{j}{Nq}, \frac{j+1}{Nq}] \subset \mathbb{S}^{1}$. Proposition \ref{finiteconjrat} holds for the rotation $R_{\frac{p}{q}}$ with those intervals. Proposition \ref{finiteconjrat} applied to the homeomorphism $f$ provides intervals $I_{j}$ with the properties given by the Proposition. Take any homeomorphism $h$ in $\mathrm{Homeo}_{0}(\mathbb{S}^{1})$ such that, for any $j$, $h(I_{j})=I'_{j}$. Then, for any $j$, $hfh^{-1}(I'_{j}) \subset I'_{j+Np-1} \cup I'_{j+Np} \cup I'_{j+Np+1}$ and hence $d(hfh^{-1}, R_{\frac{p}{q}}) < \epsilon$.
\end{proof}

\section{Conjugacy classes: case of the disc} \label{sectdisc}

In this section, we see $\mathbb{D}^{2}$ as the unit disc in the Euclidean plane. We denote by $\mathrm{Homeo}_{0}(\mathbb{D}^{2},\partial \mathbb{D}^{2})$ the identity component of the group of homeomorphisms of the disc which pointwise fix a neighbourhood of the boundary. As a warm-up, we start with the following easy proposition. 

\begin{proposition}
Any element of $\mathrm{Homeo}_{0}(\mathbb{D}^{2}, \partial \mathbb{D}^{2})$ has conjugates arbitrarily close to the identity.
\end{proposition}

\begin{proof}
Take $\epsilon >0$ and an element $f$ of $\mathrm{Homeo}_{0}(\mathbb{D}^{2}, \partial \mathbb{D}^{2})$. Let $B$ be a closed disc which is contained in the interior of the disc $\mathbb{D}^{2}$ and whose interior contains the support of $f$. Finally, let $h$ be a homeomorphism in $\mathrm{Homeo}_{0}(\mathbb{D}^{2}, \partial \mathbb{D}^{2})$ which sends the disc $B$ to a disc $B'$ whose diameter is smaller than $\epsilon$. As the homeomorphism $hfh^{-1}$ is supported in $B'$, this homeomorphism is $\epsilon$-close to the identity.
\end{proof}

The goal of this section is to prove Theorem \ref{conjugacyclassdisc}

In the proof of the theorem, we need the following easy lemma.

\begin{lemma} \label{extension}
Let $\varphi$ be an orientation preserving homeomorphism of the circle $\partial \mathbb{D}^{2}$. There exists a homeomorphism $h$ in $\mathrm{Homeo}_{0}(\mathbb{D}^{2})$ such that $h_{|\partial \mathbb{D}^{2}}= \varphi$. 
\end{lemma}

\begin{proof}
Denote by $\left\|. \right\|$ the Euclidean norm on $\mathbb{R}^2$. We see $\mathbb{D}^2$ as the unit disc in $\mathbb{R}^2$. Take the homeomorphism defined by 
$$\begin{array}{rrcl}
h: & \mathbb{D}^{2} & \rightarrow & \mathbb{D}^{2} \\
 & x \neq 0 & \mapsto & \left\|x \right\| \varphi(\frac{x}{\left\| x \right\|}) \\
  & 0 & \mapsto & 0
 \end{array}
.$$
An isotopy between $\varphi$ and the identity provides an isotopy between $h$ and the identity.
\end{proof}

\begin{proof}[Proof of Theorem \ref{conjugacyclassdisc}]
We will distinguish two cases depending on whether the number rotation number $\rho(f)$ of $f$ on the boundary $\partial \mathbb{D}^2$ is rational or not.

\noindent \underline{First case:} Suppose that $\alpha=\rho(f)$ is irrational.

Fix $\epsilon>0$. Denote by $\gamma$ the oriented arc
$$ \begin{array}{rcl}
[0,1] & \rightarrow & \mathbb{D}^{2} \subset \mathbb{R}^2 \\
t & \mapsto & (t,0)
\end{array}
.$$
Let $N$ and $M$ be integers. Denote by $R_{\alpha}^{q_{1}}(\gamma), R_{\alpha}^{q_{2}}(\gamma),\ldots, R_{\alpha}^{q_{N}}(\gamma)$ the curves $R_{\alpha}(\gamma), R_{\alpha}^{2}(\gamma),\ldots, R_{\alpha}^{N}(\gamma)$
ordered in such a way that $q_{1}=1$ and for any index $i \in \mathbb{Z}/n \mathbb{Z}$ $R_{\alpha}^{q_{i+1}}(\gamma)$ is the curve in the set 
$$\left\{ R_{\alpha}^{k}(\gamma), \left\{ \begin{array}{l} 1 \leq k \leq N \\ k \neq q_{i} \end{array} \right. \right\}$$
which is immediately on the right of $R_{\alpha}^{q_{i}}(\gamma)$.

For any $i \in \mathbb{Z} / N \mathbb{Z}$ and any $0 \leq j \leq M-1$, denote by $S_{i,j}$ the square bounded on the left by $R_{\alpha}^{q_{i}}(\gamma)$, on the right by $R_{\alpha}^{q_{i+1}}(\gamma)$, on the bottom by the circle $C_{j}=\left\{ \left\| x \right\|= \frac{j}{M} \right\}$ and on top by the circle $C_{j+1}=\left\{ \left\| x \right\|= \frac{j+1}{M} \right\}$. We take $N$ and $M$ sufficiently large so that the two following properties hold:
\begin{enumerate}
\item For any $i \in \mathbb{Z} / N \mathbb{Z}$ and any $1 \leq j \leq M-1$, the square $S_{i,j}$ has a diameter smaller than $\epsilon$.
\item The diameter of the set $\left\{ \left\| x \right\| \leq \frac{1}{M} \right\}$ is smaller than $\epsilon$.
\end{enumerate}

We will find a similar decomposition for $f$ in order to build our conjugation. Denote by $x$ the point $(1,0)$. Use Proposition \ref{finiteconj} and Lemma \ref{extension} to find a homeomorphism $h$ in $\mathrm{Homeo}_{0}(\mathbb{D}^{2})$ such that, for any $0 \leq k \leq N$,
$$h(f^{k}(x))=R_{\alpha}^{k}(h(x))=R_{\alpha}^{k}(x).$$
We want to build a homeomorphism $h'$ such that the homeomorphism $h'hfh^{-1}h'^{-1}$ is close to the rotation $R_{\alpha}$.

Let $\delta: [0,1] \rightarrow \mathbb{D}^{2}$ be an embedded arc with the following properties:
\begin{enumerate}
\item$\delta([0,1)) \cap \partial \mathbb{D}^{2} = \emptyset$.
\item $\delta(1)=x$.
\item The arcs $\delta_{k}=(hfh^{-1})^{k}(\delta)$, for $0 \leq k \leq N$, are pairwise disjoint.
\end{enumerate}
Observe that any small enough arc satisfying the two first properties also satisfies the third property. Observe also that the arcs $hf^{k}h^{-1}(\delta)$, for $0 \leq k \leq N$, are in the same order as the arcs $R_{\alpha}^{k}(\gamma)$. Indeed, for any $k$ the arc $hf^{k}h^{-1}(\delta)$ has the same endpoint as the arc $R_{\alpha}^{k}(\gamma)$. These arcs $hf^{k}h^{-1}(\delta)$ will be sent to the arcs $R_{\alpha}^{k}(\gamma)_{[\frac{1}{M},1]}$ under $h'$. 

Now, we construct the curves which will be sent to the circles $C_{i}$ under $h'$. Let $C'_{1}$ be a simple loop $\mathbb{S}^{1} \rightarrow \mathbb{D}^{2}$ contained in the interior of the disc which contains the points $(hfh^{-1})^{k}(\delta(0))$ for $0 \leq k \leq N$ and which does not contain any other point of the arcs $(hfh^{-1})^{k}(\delta)$, for $0 \leq k \leq N$. We can then construct by induction a family of simple loops $(C'_{i})_{1 \leq i \leq M}$ with the following properties:
\begin{enumerate}
\item For any $1 \leq i < j \leq M$, the loops $C'_{j}$ and $hfh^{-1}(C'_{j})$ are disjoint from the loops $C'_{i}$ and $hfh^{-1}(C'_{i})$ and lie above $C'_{i}$ and $hfh^{-1}(C'_{i})$ (\emph{i.e.} they belong to the same connected component of $\mathbb{D}^{2}-C'_{i}$ and $\mathbb{D}^{2}-hfh^{-1}(C'_{i})$ as the boundary $\partial \mathbb{D}^{2}$).
\item For any $i$ and $j$ with $0 \leq i \leq N$ and $1 \leq j \leq M$, each loop $C'_{j}$ meets each of the arcs $\delta_{i}$ in exactly one point.  
\end{enumerate}

These properties enable us to construct a homeomorphism $h'$ in $\mathrm{Homeo}_{0}(\mathbb{D}^{2})$ with the following properties.
\begin{enumerate}
\item For any $0 \leq i \leq N$, $h'(hf^{i}h^{-1}(\delta))=R_{\alpha}^{i}(\gamma_{|[1,M]})$.
\item  For any $1 \leq j \leq M$, $h'(C'_{j})=C_{j}$.
\end{enumerate}

Denote by $S$ a connected component of the complement of $\bigcup \limits _{j=1}^{m} C_{j} \cup \bigcup \limits _{i=0}^{N-1}R_{\alpha}^{i}(\gamma([\frac{1}{M},1]))$ which is different from the disc $D_{\frac{1}{M}}$ of center $0$ and radius $\frac{1}{M}$. Then there exist $i$ and $j$ such that $R_{\alpha}(S)=S_{i,j}$. By construction, the image under $h'hf(h'h)^{-1}$ of $S$ is contained in
\begin{itemize}
\item $S_{i,j} \cup S_{i,j-1} \cup S_{i,j+1}$ if $j>1$.
\item $S_{i,1} \cup D_{\frac{1}{M}} \cup S_{i,2}$ if $j=1$.
\end{itemize}
Moreover, the homeomorphism $h'hf(h'h)^{-1}$ sends the disc $D_{\frac{1}{M}}$ to $D_{\frac{1}{M}} \cup \bigcup_{i} S_{i,1}= R_{\alpha}(D_{\frac{1}{M}}) \cup \bigcup_{i} S_{i,1}$. As the sets $D_{\frac{1}{M}}$ and $S_{i,j}$ have a diameter smaller than $\epsilon$, we deduce that
$$d(R_{\alpha}, h'hf(h'h)^{-1}) <2 \epsilon,$$
where $d$ denotes the uniform distance.

\noindent \underline{Second case:} Case where $\rho(f)=\frac{p}{q}$ is rational. This case is similar to the first one: we will skip some details. For notational reasons, the unit circle is identified with $\mathbb{R} / \mathbb{Z}$. Fix large integers $N>0$ and $M>0$. First, use Proposition \ref{finiteconjrat} to obtain intervals $(I_{i})_{0 \leq i \leq Np}$ corresponding to $f_{|\partial \mathbb{D}^{2}}$. Then use Lemma \ref{extension} to obtain a homeomorphism $h$ of the disc  which sends the interval $I_{j}$ to the interval $[\frac{i}{Nq},\frac{i+1}{Nq}]$ of the circle $\partial \mathbb{D}^{2}$. 

For any $0 \leq i \leq Nq-1$, take a small arcs $\delta_{i}: [0,1] \rightarrow \mathbb{D}^{2}$ which touch $\partial \mathbb{D}^{2}$ only at the point $\delta_{i}(1)=\frac{i}{Nq} \in \partial \mathbb{D}^{2}$. Choose these arcs so that they are pairwise disjoint. For any $1 \leq j \leq M$, take a loop $C'_{j}$ which meets each $\delta_{i}$ in only one point. Construct them so that the following properties hold.
\begin{enumerate}
\item $C'_{M}= \partial \mathbb{D}^{2}$.
\item The loop $C'_{1}$ meets each curve $\delta_{i}$ at $\delta_{i}(0)$.
\item For any $j \geq 1$, the loop $C'_{j}$ is above $C'_{j-1}$ (disjoint from $C'_{j-1}$ and in the same connected component of $\mathbb{D}^{2}-C'_{j}$ as $\partial \mathbb{D}^{2}$).
\item For any $1 \leq j \leq M-1$, the curve $hfh^{-1}(C'_{j})$ is disjoint from $C'_{j-1}$ and $C'_{j+1}$.
\end{enumerate}

Construct then a homeomorphism $h'$ in $\mathrm{Homeo}_{0}(\mathbb{D}^{2})$ with the following properties.
\begin{enumerate}
\item It sends each loop $C'_{j}$ to the circle of radius $\frac{j}{M}$.
\item It sends each curve $\delta_{i}$ to the straight line contained in a radius of the unit disc joining the circle of radius $\frac{1}{M}$ to the point $\frac{i}{Nq}$ of the circle $\partial \mathbb{D}^{2}$.
\end{enumerate}

One can check that the homeomorphism $h'hf(h'h)^{-1}$ is close to the rotation of angle $\rho(f)$ if $N$ and $M$ are chosen sufficiently large.
\end{proof}

\section{Conjugacy classes: case of the annulus} \label{sectannulus}

This section is devoted to the proof of Theorem \ref{conjugacyclassannulus}. This proof uses the notion of rotation set of a homeomorphism of the annulus isotopic to the identity. For more background on this notion, see the article \cite{MZ} by Misiurewicz and Ziemian. In the quoted article, the notion is introduced in the case of homeomorphisms of the torus but everything carries over in the (easier) case of the annulus. For any homeomorphism $f$ in $\mathrm{Homeo}_{0}(\mathbb{A})$, we denote by $\rho(f)$ its rotation set.

A simple curve $\gamma : [0,1] \rightarrow \mathbb{S}^{1} \times [0,1]= \mathbb{A}$ (respectively $\gamma : [0,1] \rightarrow \mathbb{R} \times [0,1]$) is said to \emph{join the two boundary components of the annulus (respectively the strip)} if:
\begin{itemize}
\item $\gamma(0) \in \mathbb{S}^{1} \times \left\{ 0 \right\}$ and $\gamma(1) \in \mathbb{S}^{1} \times \left\{1 \right\}$ (respectively $\gamma(0) \in \mathbb{R} \times \left\{ 0 \right\}$ and $\gamma(1) \in \mathbb{R} \times \left\{1 \right\}$).
\item $\gamma((0,1)) \subset \mathbb{S}^{1} \times (0,1)$ (respectively $\gamma((0,1)) \subset \mathbb{R} \times (0,1)$).
\end{itemize}

Given a simple curve $\gamma$ which joins the two boundary components of the strip $\mathbb{R} \times [0,1]$, the set $\mathbb{R}\times [0,1]- \gamma([0,1])$ consists of two connected components. As the curve $\gamma$ is oriented by the parametrization, it makes sense to say that one of them is on the right of $\gamma$ and the other one is on the left of $\gamma$.

\begin{definition}
Take a simple curve $\gamma$ which joins the two boundary components of the strip $\mathbb{R} \times [0,1]$. A subset of $\mathbb{R} \times [0,1]$ is said to \emph{lie strictly on the right} (respectively \emph{strictly on the left}) of the curve $\gamma$ if it is contained on the connected component of $\mathbb{R} \times [0,1]-\gamma$ on the right (respectively on the left) of $\gamma$.
\end{definition}

\begin{definition}
Take three pairwise disjoint simple curves $\gamma_{1}$, $\gamma_{2}$ and $\gamma_{3}$ which join the two boundary components of the annulus. We say that the curve $\gamma_{2}$ lies strictly between the curves $\gamma_{1}$ and $\gamma_{3}$ if the following property is satisfied. There exists lifts $\tilde{\gamma}_{1}$ and $\tilde{\gamma}_{2}$ to the strip of respectively $\gamma_{1}$ and $\gamma_{2}$ such that:
\begin{enumerate}
\item The curve $\tilde{\gamma}_{2}$ lies strictly on the right of $\tilde{\gamma_{1}}$.
\item For any lift $\tilde{\gamma}_{3}$ of the curve $\gamma_{3}$ which lies strictly on the right of $\tilde{\gamma}_{1}$, the curve $\tilde{\gamma}_{2}$ lies strictly on the left of $\tilde{\gamma}_{3}$.
\end{enumerate}
\end{definition}

Notice that a curve which lies strictly between $\gamma_{1}$ and $\gamma_{3}$ does not lie strictly between $\gamma_{3}$ and $\gamma_{1}$.

\begin{proposition} \label{pavageverticalanneau} (see Figure \ref{fig})
Let $f$ be a homeomorphism in $\mathrm{Homeo}_{0}(\mathbb{A})$ and $p$ and $q$ be integers such that either $q>0$, $0<p<q$ and $p$ and $q$ are mutually prime or $p=0$ and $q=1$. Let us fix an integer $n>1$. Suppose that $\rho(f)= \left\{ \frac{p}{q} \right\}$. Then there exists a family of pairwise disjoint simple curves $(\gamma_{i})_{i \in \mathbb{Z} / nq \mathbb{Z}}$ which join the two boundary components of the annulus such that, for any index $i$:
\begin{enumerate}
\item The curve $\gamma_{i}$ lies strictly between the curves $\gamma_{i-1}$ and $\gamma_{i+1}$.
\item The curve $f(\gamma_{i})$ lies strictly between the curves $\gamma_{i+np-1}$ and $\gamma_{i+np+1}$.
\end{enumerate}
\end{proposition}

\begin{rema}
In the case of the rotation $R_{\frac{p}{q}}$, note that it suffices to take $\gamma_{i}(t)=(\frac{i}{nq},t)$.
\end{rema}

\begin{figure}[ht]
\begin{center}
\includegraphics[scale=0.5]{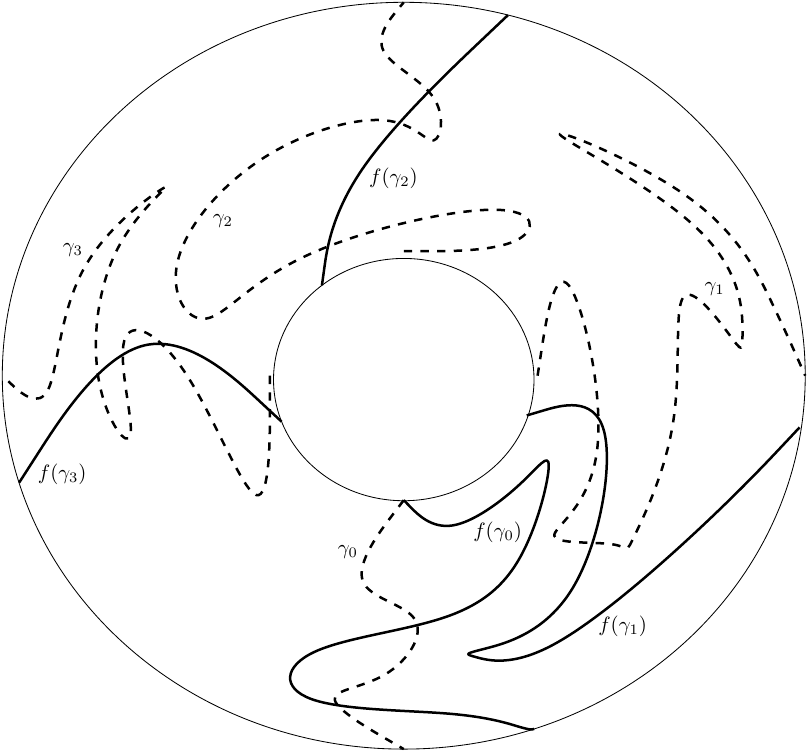}
\end{center}
\caption{Illustration of Proposition \ref{pavageverticalanneau} in the case $p=0$, $q=1$ and $n=4$}
\label{fig}
\end{figure}

For technical reasons, it is more convenient to prove the following stronger proposition.

\begin{proposition}
Let $f$ be a homeomorphism in $\mathrm{Homeo}_{0}(\mathbb{A})$ and $p$ and $q$ be integers such that either $q>0$, $0<p<q$ and $p$ and $q$ are mutually prime or $p=0$ and $q=1$. Suppose that $\rho(f)= \left\{ \frac{p}{q} \right\}$. Let us fix integers $n>0$ and $N>0$. Then there exists a family of pairwise disjoint simple curves $(\gamma'_{i})_{ i \in \mathbb{Z} / nq \mathbb{Z}}$ which join the two boundary components of the annulus such that, for any index $i$ and any integer $0\leq k \leq N$:
\begin{enumerate}
\item If $n \neq 1$ or $q \neq 1$, the curve $f^{k}(\gamma'_{i})$ lies strictly between the curves $\gamma'_{i+knp-1}$ and $\gamma'_{i+knp+1}$.
\item If $n=q=1$, any lift of the curve $f^{k}(\gamma'_{1})$ meets at most one lift of the curve $\gamma'_{1}$.
\end{enumerate}
\end{proposition}

The case $N=1$ of this proposition yields directly Proposition \ref{pavageverticalanneau}.

\begin{proof}
We say that a finite sequence of curves $(\gamma'_{i})_{ i \in \mathbb{Z} / nq \mathbb{Z}}$ satisfies property $P(N,n)$ if it satisfies the conclusion of the proposition. We prove by induction on $n$ that, for any $N$, there exist curves $(\gamma'_{i})_{i \in \mathbb{Z}/qn \mathbb{Z}}$ which satisfy property $P(N,n)$ and such that, for any index $n \leq i <nq$, $\gamma'_{i}= f(\gamma'_{i-np})$.

We first check the case where $n=1$, which is actually the most difficult one. The proof in this case relies on the following lemma due to Béguin, Crovisier, Le Roux and Patou (see \cite{BCLRP}, Proposition 3.1). Let 
$$\mathrm{Homeo}_{\mathbb{Z}}(\mathbb{R})=\left\{ f \in \mathrm{Homeo}(\mathbb{R} \times [0,1]), \forall(x,y) \in \mathbb{R} \times [0,1], f(x+1,y)=f(x,y)+(1,0) \right\}.$$

\begin{lemma} \label{commutinghomeos}
Let $F_{1}, \ldots, F_{l}$ be pairwise commuting homeomorphisms in $\mathrm{Homeo}_{\mathbb{Z}}(\mathbb{R} \times [0,1])$. Suppose that, for any index $i$, $\rho(F_{i}) \subset (0,+\infty)$. Then there exists an essential simple curve $\tilde{\gamma}: [0,1] \rightarrow \mathbb{R} \times [0,1]$ which joins the two boundary components of the strip and satisfies the following property. For any index $i$, the curve $F_{i}(\tilde{\gamma})$ lies strictly on the right of the curve $\tilde{\gamma}$.
\end{lemma}

Fix an integer $N>0$. We denote by $\tilde{f}$ the lift of the homeomorphism $f$ such that $\rho(\tilde{f})= \left\{ \frac{p}{q} \right\}$ and by $T$ the translation of $\mathbb{R} \times [0,1]$ defined by $(x,t) \mapsto (x+1,t)$. Consider the unique permutation $\sigma$ of $\llbracket 1, q-1 \rrbracket= \left\{ 1, \ldots, q-1 \right\}$, such that there exists a finite sequence of integers $(t(i))_{1 \leq i \leq q-1}$ with
$$ 0 < \sigma(1) \frac{p}{q} +t(1) < \sigma(2) \frac{p}{q} + t(2) < \ldots < \sigma(q-1) \frac{p}{q} +t(q-1) <1.$$
Notice that $\sigma(i) \frac{p}{q}+t(i)= \frac{i}{q}$. Hence $\sigma(i)$ is equal to $\frac{i}{p}$ mod $q$ (observe that $p$ is invertible in $\mathbb{Z}/ q \mathbb{Z}$ as the integers $p$ and $q$ are mutually prime). Equivalently, the integer $\sigma^{-1}(i)$ is the unique representative in $\llbracket 1,q-1 \rrbracket$ of $ip$ mod $q$. To simplify notation, let $\sigma(0)=0$, $\sigma(q)=0$, $t(0)=0$ and $t(q)=1$. Let $M$ be any integer greater than $\frac{N}{q}$. We now apply Lemma \ref{commutinghomeos} to the homeomorphisms of one of the following forms, for $0 \leq j \leq M$ and $0 \leq i \leq q-1$:
\begin{enumerate}
\item $T^{t(i+1)-jp}\tilde{f}^{\sigma(i+1)+jq}T^{-t(i)}\tilde{f}^{-\sigma(i)}$ whose rotation set is 
$$\left\{ t(i+1)+ \sigma(i+1)\frac{p}{q}-t(i)-\sigma(i) \frac{p}{q} \right\} \subset (0, + \infty).$$
\item $T^{t(i+1)}\tilde{f}^{\sigma(i+1)}T^{-t(i)+jp}\tilde{f}^{-\sigma(i)-jq}$ whose rotation set is 
$$\left\{ t(i+1)+ \sigma(i+1)\frac{p}{q}-t(i)-\sigma(i) \frac{p}{q} \right\} \subset (0, + \infty).$$
\end{enumerate}
Lemma \ref{commutinghomeos} provides a simple curve $\tilde{\gamma}': [0,1] \rightarrow [0,1] \times \mathbb{R}$ such that, for any $j \in \llbracket 0,M \rrbracket$,  and any $i \in \llbracket 0,q-1 \rrbracket$:
\begin{enumerate}
\item The curve $T^{t(i+1)-jp} \tilde{f}^{\sigma(i+1)+jq}(\tilde{\gamma}')$ lies strictly on the right of the curve $T^{t(i)}\tilde{f}^{\sigma(i)}(\tilde{\gamma}')$.
\item The curve $T^{t(i+1)} \tilde{f}^{\sigma(i+1)}(\tilde{\gamma}')$ lies strictly on the right of the curve $T^{t(i)-jp}\tilde{f}^{\sigma(i)+jq}(\tilde{\gamma}')$.
\end{enumerate}
In particular, by the first property above with $j=0$, the curve $T^{1}(\tilde{\gamma}')$ lies strictly on the right of the curve $T^{t(q-1)}\tilde{f}^{\sigma(q-1)}(\tilde{\gamma}')$ which lies itself strictly on the right of the curve $T^{t(q-2)}\tilde{f}^{\sigma(q-2)}(\tilde{\gamma}')$ and so forth. Hence the curve $T^{1}(\tilde{\gamma}')$ lies strictly on the right of the curve $\tilde{\gamma}'$: the projection $\gamma'$ of the curve $\tilde{\gamma}'$ on the annulus is a simple curve. We set $\gamma'_{i}= f^{\sigma(i)}(\gamma')$. We have seen that the curve $T^{t(i)}\tilde{f}^{\sigma(i)}(\tilde{\gamma}')$ is the unique lift of the curve $\gamma'_{i}$ which lies between the curves $\tilde{\gamma}'$ and $T^{1}( \tilde{\gamma}')$. Let us check that the curves $\gamma'_{i}$ satisfy the desired properties.  

Fix $i' \in \llbracket 0, q-1 \rrbracket$ and $k \in \llbracket 0,N \rrbracket$. Perform the Euclidean division of $k+\sigma(i')$ by $q$: $k+\sigma(i')=jq+r$. By the two above properties, the curve $f^{jq+r}(\gamma')$ lies strictly between the curves $\gamma'_{\sigma^{-1}(r)-1}$ and $\gamma'_{\sigma^{-1}(r)+1}$. To see this, if $r \neq 0$, apply the first property for $i=\sigma^{-1}(r)-1$ and the second property for $i=\sigma^{-1}(r)$, and, if $r=0$, apply the first property for $i=q-1$ and the second property for $i=0$. Now, remember that, modulo $q$:
$$\begin{array}{rcl}
 \sigma^{-1}(r) & = & \sigma^{-1}(k+\sigma(i')-jq) \\
 & = & (k+ \sigma(i'))p \\
 & = & kp+i'.
\end{array}
$$
This proves the proposition when $n=1$.

Suppose that there exist curves $(\alpha_{i})_{ i \in \mathbb{Z}/nq\mathbb{Z}}$ which satisfy $P(2Nq,n)$ and such that, for any index $n \leq i <nq$, $\alpha_{i}= f(\alpha_{i-np})$. Let us construct curves $(\gamma'_{i})_{i \in \mathbb{Z}/(n+1)q \mathbb{Z}}$ which satisfy $P(Nq,n+1)$ (hence $P(N,n+1)$) and such that for any index $n+1 \leq i <(n+1)q$, $\gamma'_{i}= f(\gamma'_{i-np})$. If $i$ is not equal to $1$ mod $n+1$, the curve $\gamma'_{i}$ is one of the curves $f^{Nq}(\alpha_{j})$. More precisely, write the Euclidean division of $i$ by $n+1$: $i=l(i)(n+1)+r(i)$. If $r(i) >1$, then $\gamma'_{i}=f^{Nq}(\alpha_{j})$, with $j=l(i)n+r(i)-1$. If $r(i)=0$, then $\gamma'_{i}=f^{Nq}(\alpha_{j})$, with $j=l(i)n$. We now build the curve $\gamma'_{1}$.

Notice that, for any integers $-N \leq k,k' \leq N$, $f^{(N+k)q}(\alpha_{0}) \cap f^{(N+k')q}(\alpha_{1})= \emptyset$. Indeed, recall that, by Property $P(2Nq,n)$, the curves of the form $f^{lq}(\alpha_{0})$, with $0 \leq l \leq 2N$, lie strictly between the curves $\alpha_{-1}$ and $\alpha_{1}$. Likewise, the curves of the form $f^{lq}(\alpha_{1})$, with $0 \leq l \leq 2N$, lie strictly between the curves $\alpha_{0}$ and $\alpha_{2}$. Moreover, the intersection $f^{(N+k)q}(\alpha_{0}) \cap f^{(N+k')q}(\alpha_{1})$ is equal to $f^{(N+k)q}(\alpha_{0}\cap f^{(k'-k)q}(\alpha_{1}))$ or $f^{(N+k')q}(f^{(k-k')q}(\alpha_{0}) \cap \alpha_{1})$ and, among the integers $k-k'$ and $k'-k$, one is nonnegative and smaller than or equal to $2N$. 

Hence there exists a simple curve $\gamma'_{1}: [0,1] \rightarrow \mathbb{A}$ such that:
\begin{enumerate}
\item $\gamma'_{1}(0) \in \mathbb{S}^{1} \times \left\{ 0 \right\}$, $\gamma'_{1}(1) \in \mathbb{S}^{1} \times \left\{ 1 \right\}$ and $\gamma'_{1}((0,1)) \subset \mathbb{S}^{1} \times (0,1)$.
\item For any integers $k,k' \in \llbracket -N,N \rrbracket$, the curve $\gamma'_{1}$ lies strictly between the curves $f^{(N+k)q}(\alpha_{0})=f^{kq}(\gamma'_{0})$ and $f^{(N+k')q}(\alpha_{1})=f^{k'q}(\gamma'_{2})$.
\end{enumerate} 

By the second property above, for any integer $k \in \llbracket 0,N \rrbracket$, the curve $f^{kq}(\gamma'_{1})$ lies strictly between the curves $\gamma'_{0}$ and $\gamma'_{2}$. Moreover, the curves of the form $f^{kq}(\gamma'_{0})$, with $0 \leq k \leq N$ lie strictly between the curves $\gamma'_{-1}$ and $\gamma'_{1}$ and the curves of the form $f^{kq}(\gamma'_{2})$, with $0 \leq k \leq N$ lie strictly between the curves $\gamma'_{1}$ and $\gamma'_{3}$. If $q=1$, we have proved that the finite sequence $(\gamma'_{i})_{ i \in \mathbb{Z}/(n+1)q \mathbb{Z}}$ satisfies $P(N,n+1)$.

Suppose now that $q \neq 1$. For any index $i \neq 1$ with $r(i)=1$, there exists a unique integer $j \in \llbracket 1,q-1 \rrbracket$ such that $i=1+j(n+1)p$. Set $\gamma'_{i}=f^{j}(\gamma'_{1})$. As $\gamma'_{i-1}=f^{j}(\gamma'_{0})$ and $\gamma'_{i+1}=f^{j}(\gamma'_{2})$ by induction hypothesis, it is easy to check that the finite sequence $(\gamma'_{i})_{i \in \mathbb{Z}/(n+1)q \mathbb{Z}}$ satisfies $P(Nq,n+1)$.
\end{proof}

\begin{proof} [Proof of Theorem \ref{conjugacyclassannulus}]
In the case where $\alpha$ is irrational, the theorem is Corollary 1.2 in \cite{BCLRP}. Suppose that $\alpha= \frac{p}{q}$, where $p$ and $q$ are integers with either $q=1$ and $p=0$ or $q >0$ and $0<p <q$. Fix large integers $N,N' >0$. Apply Proposition \ref{pavageverticalanneau} to the homeomorphism $f$ with $n=N$: this proposition provides curves $(\gamma_{i})_{ i \in \mathbb{Z}/Nq \mathbb{Z}}$. Consider a finite sequence $(\alpha_{j})_{j \in \llbracket 0,N' \rrbracket}$ of pairwise disjoint loops $\mathbb{S}^{1} \rightarrow \mathbb{A}$ such that:
\begin{enumerate}
\item For any $t \in \mathbb{S}^{1}$, $\alpha_{0}(t)=(t,0)$ and $\alpha_{N'}(t)=(t,1)$.
\item The loops $\alpha_{j}$ are homotopic to $\alpha_{0}$.
\item For any index $1 \leq j <N'$, the loops $\alpha_{j}$ and $f(\alpha_{j})$ lie strictly between the curves $\alpha_{j+1}$ and $\alpha_{j-1}$.
\item For any indices $i$ and $j$, the loop $\alpha_{i}$ meets the curve $\gamma_{j}$ in only one point.
\end{enumerate}

Such curves can be built by induction on $N'$.

\begin{figure}[ht]
\begin{center}
\includegraphics[scale=0.70]{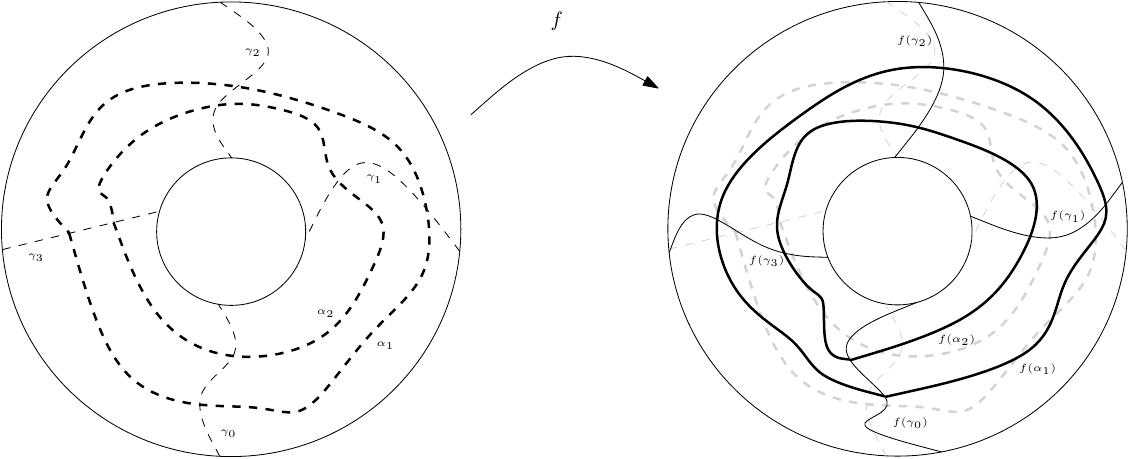}
\end{center}
\caption{Action of the homeomorphism $f$ on the curves $\gamma_{i}$ and $\alpha_{j}$ in the case $\alpha =0$, $N=4$ and $N'=3$}
\end{figure}

\begin{figure}[ht]
\begin{center}
\includegraphics[scale=0.5]{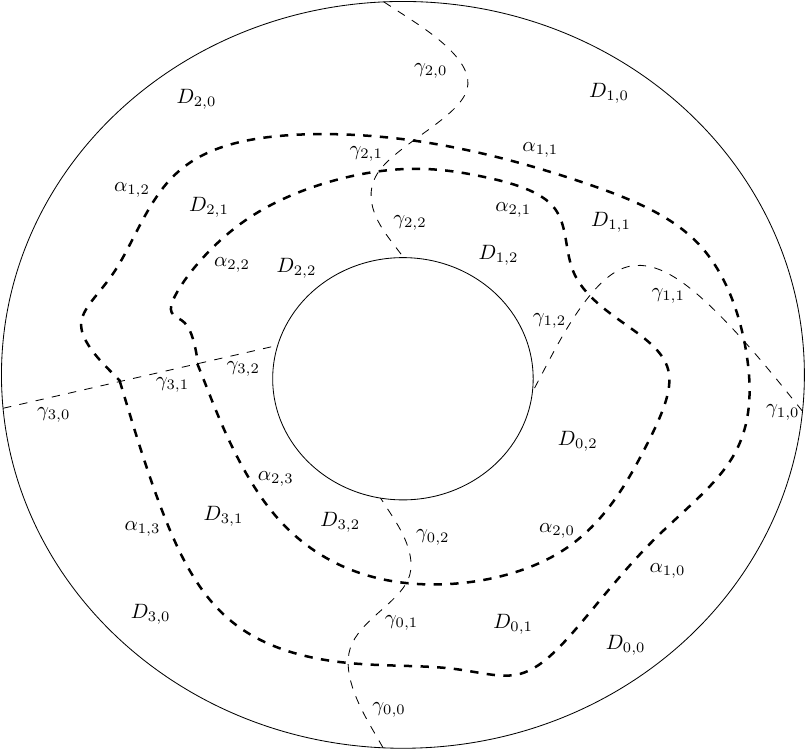}
\end{center}
\caption{Notation for the proof of the theorem in the case $\alpha =0$, $N=4$ and $N'=3$}
\label{notationstheorem}
\end{figure}

Let us introduce some notation (see Figure \ref{notationstheorem}). For any $i \in \mathbb{Z}/Nq \mathbb{Z}$, let $\gamma'_{i}$ be the curve $[0,1] \rightarrow \mathbb{A}$ defined by: $\gamma'_{i}(t)=(\frac{i}{Nq},t)$ and, for any $j \in \llbracket0, N' \rrbracket$, let $\alpha'_{j}$ be the loop $\mathbb{S}^{1} \rightarrow \mathbb{A}$ defined by $\alpha'_{j}(t)=(t,\frac{j}{N'})$. For $i$ in $\mathbb{Z}/Nq \mathbb{Z}$ and $j$ in $\llbracket 0, N' \rrbracket$, denote by $\alpha_{j,i}$ (respectively $\alpha'_{j,i}$) the closure of the connected component of $\alpha_{j}- \cup_{i'} \gamma_{i'}$ (respectively $\alpha'_{j}- \cup_{i'} \gamma'_{i'}$) which lies strictly between the curves $\gamma_{i}$ and $\gamma_{i+1}$ (respectively between the curves $\gamma'_{i}$ and $\gamma'_{i+1}$). Notice that, for any $j$, the loop $\alpha_{j}$ is the concatenation of the $\alpha_{j,i}$'s. Similarly, for any $i \in \mathbb{Z} / Nq \mathbb{Z}$ and any $0 \leq j \leq N'-1$, denote by $\gamma_{i,j}$ (respectively $\gamma'_{i,j}$) the closure of the connected component of $\gamma_{i}- \cup_{j'} \alpha_{j'}$ (respectively of $\gamma'_{i}- \cup_{j'} \alpha'_{j'}$) which lies strictly between the curves $\alpha_{j+1}$ and $\alpha_{j}$ (respectively between the curves $\alpha'_{j+1}$ and $\alpha'_{j}$). Finally, for $i \in \mathbb{Z} / Nq \mathbb{Z}$ and $0\leq j \leq N'-1$, we denote by $D_{i,j}$ (respectively $D'_{i,j}$) the topological closed disc whose boundary is the Jordan curve $\gamma_{i,j} \cup \alpha_{j,i} \cup \gamma_{i+1,j} \cup \alpha_{j+1,i}$ (respectively $\gamma'_{i,j} \cup \alpha'_{j,i} \cup \gamma'_{i+1,j} \cup \alpha'_{j+1,i}$). Note that $D'_{i,j}=[\frac{i}{Nq}, \frac{i+1}{Nq}] \times [\frac{j}{N'}, \frac{j+1}{N'}]$. The discs $D_{i,j}$ as well as $D'_{i,j}$ have pairwise disjoint interiors and cover the annulus $\mathbb{A}$. 

Consider a homeomorphism $h$ of the annulus which sends, for any $(i,j)$, the path $\gamma_{i,j}$ onto the path $\gamma'_{i,j}$ and the path $\alpha_{j,i}$ onto the path $\alpha'_{j,i}$. Such a homeomorphism exists thanks to the Schönflies theorem and sends each disk $D_{i,j}$ onto the corresponding disk $D'_{i,j}$.

By the properties of the curves $\alpha_{i}$ and $\gamma_{j}$, for any $(i,j)$, the loop $f(\partial D_{i,j})$ lies strictly between the curves $\alpha_{\min(j+2,N')}$ and $\alpha_{\max(j-1,0)}$. and strictly between the curves $\gamma_{i+Np-1}$ and $\gamma_{i+Np+2}$. Hence
$$f(\partial D_{i,j}) \subset \bigcup_{ \epsilon_{1}, \epsilon_{2} \in \left\{ -1,0,1 \right\}} D_{i+Np+\epsilon_{1}, j+ \epsilon_{2}},$$
where $D_{i,j}= \emptyset$ whenever $j \geq N'$ or $j <0$. Therefore
$$f( D_{i,j}) \subset \bigcup_{ \epsilon_{1}, \epsilon_{2} \in \left\{ -1,0,1 \right\}} D_{i+Np+\epsilon_{1}, j+ \epsilon_{2}}$$
and 
$$hfh^{-1}(D'_{i,j}) \subset \bigcup_{ \epsilon_{1}, \epsilon_{2} \in \left\{ -1,0,1 \right\}} D'_{i+Np+\epsilon_{1}, j+ \epsilon_{2}}.$$
Obviously
$$R_{\frac{p}{q}}(D'_{i,j})= D'_{i+Np,j} \subset \bigcup_{ \epsilon_{1}, \epsilon_{2} \in \left\{ -1,0,1 \right\}} D'_{i+Np+\epsilon_{1}, j+ \epsilon_{2}}.$$
We deduce that the uniform distance between the rotation $R_{\frac{p}{q}}$ and $hfh^{-1}$ is bounded by the supremum of the diameters of the sets  
$$\bigcup_{ \epsilon_{1}, \epsilon_{2} \in \left\{ -1,0,1 \right\}} D'_{i+Np+\epsilon_{1}, j+ \epsilon_{2}}=[\frac{i-1}{Nq}+\frac{p}{q}, \frac{i+2}{Nq}+\frac{p}{q}] \times [\frac{\max(j-1,0)}{N'}, \frac{\min(j+2,N')}{N'}].$$
This last quantity is arbitrarily small as soon as the integers $N$ and $N'$ are sufficiently large.
\end{proof}

\section{Conjugacy classes: general case} \label{sectgeneral}

In this section, we prove Theorem \ref{conjugacyclassboundary}.

We call \emph{essential arc} of the surface $S$ a simple curve $\gamma: [0,1] \rightarrow S$ up to positive reparametrization, whose endpoints lie on $\partial S$, which is not homotopic with fixed extremities to a curve contained in $\partial S$ and such that $\gamma((0,1)) \subset S- \partial S$. For any essential arc $\gamma$, by abuse of notation, we also denote by $\gamma$ the set $\gamma([0,1])$. In the case where the set $S- \gamma$ has two connected components, as the curve $\gamma$ is oriented, it makes sense to say that one of them, $C$, is on the right of $\gamma$ and the other one, $C'$, is on the left of $\gamma$. In this case, a subset $A$ of $S$ is said to lie on the right (respectively strictly on the right, on the left, strictly on the left) of the arc $\gamma$ if $A$ is contained in the closure of $C$ (respectively in $C$, in the closure of $C'$, in $C'$).

In what follows, we fix a compact surface $S$ with $\partial S \neq \emptyset$ and which is different from the annulus, the Möbius strip or the disc.

 We call \emph{maximal family of essential arcs} of $S$ a finite family $(\alpha_{i})_{1 \leq i \leq n}$ of pairwise disjoint essential arcs of $S$ such that the surface $S-(\partial S \cup \bigcup_{i} \alpha_{i})$ is homeomorphic to an open disc (see Figure \ref{maximalfamily}). Observe that the cardinality of a maximal family of essential arcs of $S$ is $1- \chi (S)$. Any family of pairwise disjoint and pairwise non homotopic (relative to $\partial S$) essential arcs whose cardinality is equal to $1 - \chi(S)$ and whose complement in $S$ is connected is a maximal family of essential arcs.

\begin{figure}[ht]
\begin{center}
\includegraphics{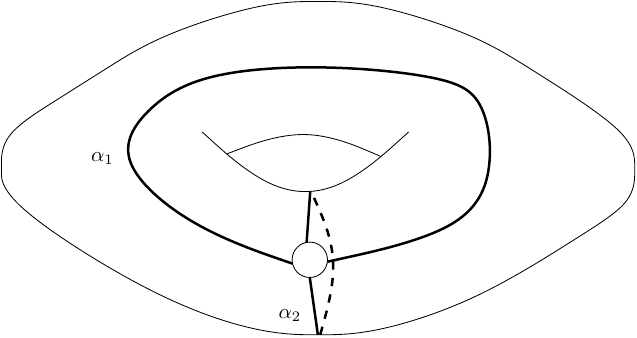}
\end{center}
\caption{A maximal family of essential arcs in the case of the punctured torus.}
\label{maximalfamily}
\end{figure}

Given such a family of essential arcs, the inclusion $i: S-(\partial S \cup \bigcup_{i} \alpha_{i}) \rightarrow S$ lifts to a map $\tilde{i}: S-(\partial S \cup \bigcup_{i} \alpha_{i}) \rightarrow \tilde{S}$. The closure of the range of such a map is a fundamental domain for the action of $\pi_{1}(S)$. We call it a \emph{fundamental domain associated to the family} $(\alpha_{i})_{1 \leq i\leq n}$. Observe that two fundamental domains associated to such a family differ by an automorphism in $\pi_{1}(S)$. Given any family $(\alpha_{i})_{1 \leq i \leq n}$ of pairwise disjoint and pairwise non homotopic (relative to $\partial S$) essential arcs such that $S- \cup_{i} \alpha_{i}$ is connected, we call fundamental domain associated to $(\alpha_{i})_{1\leq i \leq n}$ a fundamental domain associated to any maximal family of essential arcs which contains $(\alpha_{i})_{1 \leq i \leq n}$.

Fix a fundamental domain $D_{0}$ associated to some maximal family of essential arcs $(\alpha_{i,0})_{1 \leq i\leq l}$ of $S$.

\begin{proposition} \label{disjointcurvesboundary}
Let $N \geq 1$ be an integer. There exists a maximal family $(\alpha_{i})_{1 \leq i \leq l}$ of essential arcs such that the following properties hold.
\begin{enumerate}
\item Given two distinct arcs $\tilde{\alpha}, \tilde{\beta}: [0,1] \rightarrow \tilde{S}$, each of which is a lift of one of the arcs $\alpha_{i}$,  we have
$$ \forall -N \leq k \leq N, \tilde{f}^{k}(\tilde{\alpha}) \cap \tilde{\beta}= \emptyset.$$
\item There exists a fundamental domain $D$ associated to the family $(\alpha_{i})_{1 \leq i \leq l}$ and a homeomorphism $h_{0} \in \mathrm{Homeo}_{0}(S)$ such that $\tilde{h}_{0}(D)=D_{0}$.
\end{enumerate}
\end{proposition}

With this proposition, we are able to prove Theorem \ref{conjugacyclassboundary}. 

\begin{proof}[Proof of Theorem \ref{conjugacyclassboundary}]
Fix $\epsilon >0$. We will construct a homeomorphism $h$ in $\mathrm{Homeo}_{0}(S)$ such that $d(hfh^{-1},Id)= \sup_{x \in S}d(hfh^{-1}(x),x) \leq \epsilon$. Let $\varphi : D_{0} \rightarrow [0,1]^{2}$ be a homeomorphism such that the image under $\varphi$ of any essential arc contained in $\partial D_{0}$ is contained either in $[0,1] \times \left\{ 0 \right\}$ or in $[0,1] \times \left\{ 1 \right\}$ and the image under $\varphi^{-1}$ of any of the four corners of the square $[0,1] \times [0,1]$ is an endpoint of an essential arc contained in $\partial D_{0}$. Moreover, we impose that, for any $t \in [0,1]$, the points $\varphi^{-1}(t,0)$ and $\varphi^{-1}(t,1)$ do not belong to (necessarily different) lifts of the same essential arc $\alpha_{i,0}$.

Choose $L>0$ sufficiently large such that, for any $0 \leq i,j \leq L$, $\mathrm{diam}(\varphi^{-1}([i/(L+1),(i+1)/(L+1)] \times [j/(L+1),(j+1)/(L+1)])) \leq \epsilon/2$. For any $0 \leq i \leq L+1$, denote by $\tilde{\beta}_{i,0}$ (respectively $\tilde{\delta}_{i,0}$) the curve $\varphi^{-1}([0,1] \times \left\{ i/(L+1)  \right\})$ (respectively the curve $\varphi^{-1}( \left\{ i/(L+1) \right\} \times [0,1])$) oriented from the point $\varphi^{-1}(0, i/(L+1))$ to the point $\varphi^{-1}(1, i/(L+1))$ (respectively from the point $\varphi^{-1}(i/(L+1),0)$ to the point $\varphi^{-1}(i/(L+1),1)$). See Figure \ref{notationsD0}.

\begin{figure}[ht]
\begin{center}
\includegraphics{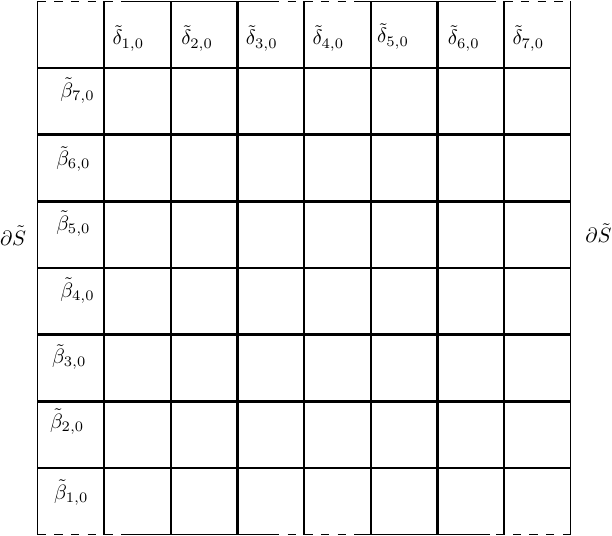}
\end{center}
\caption{Notation for the proof of Theorem \ref{conjugacyclassboundary}}
\label{notationsD0}
\end{figure}

Now, apply Proposition \ref{disjointcurvesboundary} with $N=2^{2L+1}$. We use notation from this proposition in what follows. Conjugating the homeomorphism $f$ by $h_{0}$, we can suppose that $D=D_{0}$.

Let $\tilde{\beta}_{0}= \tilde{\beta}_{0,0}$, $\tilde{\beta}_{L+1}=\tilde{\beta}_{L+1,0}$, $\tilde{\delta}_{0}=\tilde{\delta}_{0,0}$ and $\tilde{\delta}_{L+1}=\tilde{\delta}_{L+1,0}$. We will construct arcs $(\tilde{\beta}_{i})_{1 \leq i \leq  L}$ and $(\tilde{\delta}_{i})_{1 \leq i \leq L}$ such that the following properties hold.
\begin{enumerate}
\item There exists a homeomorphism $h$ in $\mathrm{Homeo}_{0}(S)$ such that $\tilde{h}(D_{0})=D_{0}$, $\tilde{h}(\tilde{\beta}_{i})= \tilde{\beta}_{i,0}$, $\tilde{h}(\tilde{\delta}_{i})= \tilde{\delta}_{i,0}$.
\item For any $1 \leq i \leq L$, the image under $\tilde{f}$ of the arc $\tilde{\beta}_{i}$ meets neither the curve $\tilde{\beta}_{i-1}$ nor the curve $\tilde{\beta}_{i+1}$.
For any $1 \leq i \leq L$, the image under $\tilde{f}$ of the arc $\tilde{\delta}_{i}$ meets neither the curve $\tilde{\delta}_{i-1}$ nor the curve $\tilde{\delta}_{i+1}$.
\item Take any essential arc $\tilde{\alpha}$ contained in $\partial D_{0}$. The image under $\tilde{f}$ or $\tilde{f}^{-1}$ of this essential arc does not meet any of the curves $\tilde{\beta}_{i}$, and any of the curves $\tilde{\delta}_{i}$ which satisfy $\tilde{\delta}_{i} \cap \tilde{\alpha}= \emptyset$.
\item Consider any essential arc $\tilde{\alpha}:[0,1] \rightarrow \tilde{S}$ contained in $\partial D_{0}$. Denote by $\gamma$ the deck transformation such that $D_{0} \cap \gamma D_{0}= \tilde{\alpha}$. Finally, let $\left\{\tilde{\alpha}(t_{i}), 1 \leq i \leq r \right\}$, be the set of points of $\tilde{\alpha}$ which belong to one of the curves $\tilde{\delta}_{i}$ or $\gamma(\tilde{\delta}_{j})$, where $t_{1} < t_{2}< \ldots <t_{r}$. Let $t_{0}=0$ and $t_{r+1}=1$. Then the following properties are satisfied.
\begin{itemize}
\item For any $l\leq r$ the image under $\tilde{f}$ of the arc $\tilde{\alpha}([t_{l},t_{l+1}])$ does not meet any of the following arcs: $\tilde{\alpha}([0,t_{l-1}])$ if $l>0$, $\tilde{\alpha}([t_{l+2},1])$ if $l<r$, the curves of the form $\tilde{\delta}_{i}$ if $\tilde{\delta}_{i} \cap \tilde{\alpha}([t_{l},t_{l+1}])= \emptyset$ and the curves of the form $\gamma(\tilde{\delta}_{j})$ if $\gamma(\tilde{\delta}_{j}) \cap \tilde{\alpha}([t_{l},t_{l+1}])= \emptyset$. 
\item For any index $i$ such that $\tilde{\alpha} \cap \tilde{\delta}_{i} \neq \emptyset$, denoting by $\tilde{\alpha}(t_{l(i)})$ the point $\tilde{\alpha} \cap \tilde{\delta}_{i}$, the image under $\tilde{f}$ of the arc $\tilde{\delta}_{i}$ does not meet any of the following arcs:  $\tilde{\alpha}([0,t_{l(i)-1}])$ if $l(i)>0$, $\tilde{\alpha}([t_{l(i)+1},1])$ if $l(i)<r+1$ and the curves of the form $\gamma(\tilde{\delta}_{j})$ if $\tilde{\delta}_{i} \cap \gamma(\tilde{\delta}_{j})= \emptyset$.
\end{itemize}
\end{enumerate}

We claim that in this case $d(hfh^{-1}, Id) \leq \epsilon$, which completes the proof of Theorem \ref{conjugacyclassboundary}.  First, let us check this claim before building the curves $\tilde{\beta}_{i}$ and $\tilde{\delta}_{i}$. In what follows, we will call square any subset of $S$ of the form $\pi(\varphi^{-1}([i/(L+1),(i+1)/(L+1)] \times [j/(L+1),(j+1)/(L+1)]))$. By the properties above, for any $0 \leq i \leq L$ and any $0 \leq j \leq L$, the image under $\tilde{h}\tilde{f}\tilde{h}^{-1}$ of any point in the square $\pi(\varphi^{-1}([i/(L+1),(i+1)/(L+1)] \times [j/(L+1),(j+1)/(L+1)]))$, which is the projection of the square delimited by $\pi(\tilde{\delta}_{i,0})$, $\pi(\tilde{\delta}_{i+1,0})$, $\pi(\tilde{\beta}_{j,0})$ and $\pi(\tilde{\beta}_{j+1,0})$, is contained in squares which meet the square $\varphi^{-1}([i/(L+1),(i+1)/(L+1)] \times [j/(L+1),(j+1)/(L+1)])$. Indeed, this is a consequence of the first three conditions above for any square which does not touch $\partial D_{0}$ (see Figure \ref{carreinterieur}) and the fourth condition ensures that this property also holds for squares which meet $\partial D_{0}$ (see Figure \ref{carreexterieur}). Any point in the union of such squares is at distance at most $\epsilon$ from any point of the square $C$, which proves the claim.

\begin{figure}[ht]
\begin{center}
\includegraphics{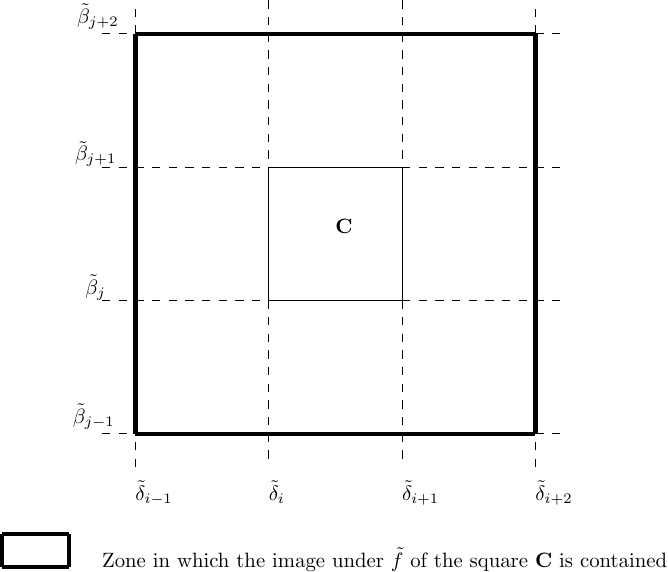}
\end{center}
\caption{The image of a square which does not meet $\partial D_{0}$}
\label{carreinterieur}
\end{figure}

\begin{figure}[ht]
\begin{center}
\includegraphics{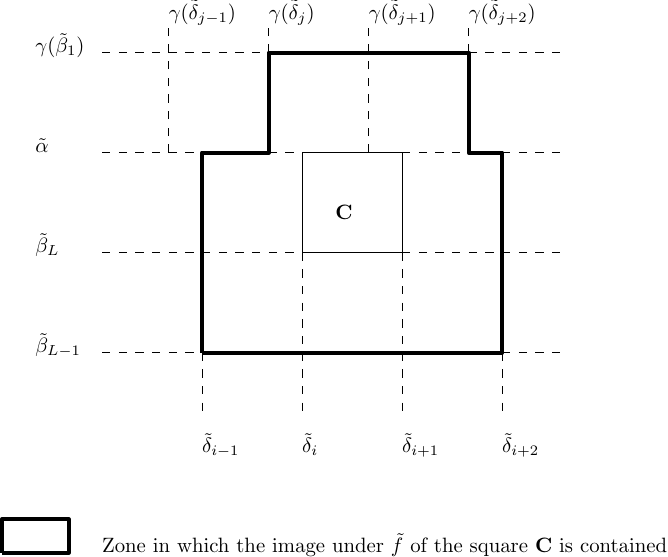}
\end{center}
\caption{The image of a square which meets $\partial D_{0}$}
\label{carreexterieur}
\end{figure}

Now let us construct the arcs $\tilde{\beta}_{i}$ and $\tilde{\delta}_{i}$. We will first build the curves $\tilde{\beta}_{i}$ by induction on $i$. More precisely, we build by induction on $1 \leq i \leq L$ a curve $\tilde{\beta}_{i}: [0,1] \rightarrow D_{0}$ such that the following properties are satisfied.
\begin{enumerate}
\item $\tilde{\beta}_{i}(0) \in \varphi^{-1}(\left\{ 0 \right\} \times [0,1])$, $\tilde{\beta}_{i}(1) \in \varphi^{-1}(\left\{ 1 \right\} \times [0,1])$ and $\tilde{\beta}_{i}((0,1)) \subset D_{0}-\partial D_{0}$.
\item The arc $\tilde{\beta}_{i}$ lies strictly on the left of the arc $\tilde{\beta}_{i-1}$ if $i>1$.
\item For any $1 \leq j < i$ and any $-2^{2L-i}\leq k,k' \leq 2^{2L-i}$ the arcs $\tilde{f}^{k}(\tilde{\beta}_{j})$ and $\tilde{f}^{k'}(\tilde{\beta}_{i})$ are disjoint.
\item For any essential arc $\tilde{\alpha}$ contained in $\partial D_{0}$ and any $-2^{2L-i}\leq k,k' \leq 2^{2L-i}$, we have $\tilde{f}^{k}(\tilde{\alpha}) \cap \tilde{f}^{k'}(\tilde{\beta}_{i})=\emptyset.$
\end{enumerate}

Recall that, by Proposition \ref{disjointcurvesboundary}, for any $-2^{2L}=-N/2 \leq k,k' \leq N/2=2^{2L}$ and any essential arcs $\tilde{\alpha} \neq \tilde{\alpha}'$ contained in $\partial D_{0}$,
$$ \tilde{f}^{k}(\tilde{\alpha}) \cap \tilde{f}^{k'}(\tilde{\alpha}')= \emptyset.$$
Hence there exists an essential arc $\tilde{\beta}_{1}:[0,1] \rightarrow D_{0}$ with the following properties.
\begin{enumerate}
\item The point $\tilde{\beta}_{1}(0)$ belongs to $\varphi^{-1}(\left\{0 \right\} \times [0,1]) \subset \partial \tilde{S}$ and the point $\tilde{\beta}_{1}(1)$ belongs to $\varphi^{-1}(\left\{1 \right\} \times [0,1]) \subset \partial \tilde{S}$.
\item $\tilde{\beta}_{1}((0,1)) \subset D_{0}- \partial D_{0}$.
\item The arc $\tilde{\beta}_{1}$ is disjoint from any of the arcs of the form $\tilde{f}^{k}(\tilde{\alpha})$, where $\tilde{\alpha}$ is any essential arc contained in $\partial D_{0}$ and $-2^{2L} \leq k \leq 2^{2L}$.  
\end{enumerate}
Observe that, for any essential arc $\tilde{\alpha}$ contained in $\partial D_{0}$ and any $-2^{2L-1}\leq k,k' \leq 2^{2L-1}$, we have $\tilde{f}^{k}(\tilde{\alpha}) \cap \tilde{f}^{k'}(\tilde{\beta}_{1})=\tilde{f}^{k'}(\tilde{f}^{k-k'}(\tilde{\alpha}) \cap \tilde{\beta}_{1})=\emptyset.$

Now suppose that we have constructed essential arcs $\tilde{\beta}_{1}, \ldots, \tilde{\beta}_{i}$ with the above properties, for some $i <L$.

We now build the arc $\tilde{\beta}_{i+1}$. By the second above property, for any essential arc $\tilde{\alpha}$ contained in $\varphi^{-1}([0,1]\times \left\{ 1 \right\})$ and any $-2^{2L-i} \leq k, k' \leq 2^{2L-i}$, the arcs $\tilde{f}^{k}(\tilde{\beta}_{i})$ and $\tilde{f}^{k'}(\tilde{\alpha})$ are disjoint. Hence there exists an essential arc $\tilde{\beta}_{i+1}: [0,1] \rightarrow D_{0}$ with the following properties.
\begin{enumerate}
\item The point $\tilde{\beta}_{i+1}(0)$ belongs to $\varphi^{-1}(\left\{0 \right\} \times [0,1]) \subset \partial \tilde{S}$ and the point $\tilde{\beta}_{i+1}(1)$ belongs to $\varphi^{-1}(\left\{1 \right\} \times [0,1]) \subset \partial \tilde{S}$.
\item $\tilde{\beta}_{i+1}((0,1)) \subset D_{0}- \partial D_{0}$.
\item The arc $\tilde{\beta}_{i+1}$ is disjoint from any of the arcs of the form $\tilde{f}^{k}(\tilde{\alpha})$, where $\tilde{\alpha}$ is any essential arc contained in $\varphi^{-1}([0,1]\times \left\{ 1 \right\})$ and $-2^{2L-i} \leq k \leq 2^{2L-i}$.
\item The arc $\tilde{\beta}_{i+1}$ is strictly on the left of any of the arcs of the form $\tilde{f}^{k}(\tilde{\beta}_{i})$, where $-2^{2L-i} \leq k \leq 2^{2L-i}$.
\end{enumerate}
It is easy to check that the arc $\tilde{\beta}_{i+1}$ satisfies the required properties.

Now, it remains to build the curves $\tilde{\delta}_{i}$ with $1 \leq i \leq L$. We build by induction on $i$ a curve $\tilde{\delta}_{i}:[0,1] \rightarrow D_{0}$ such that the following properties are satisfied.
\begin{enumerate}
\item If the point $\tilde{\delta}_{i,0}(0)$ (respectively $\tilde{\delta}_{i,0}(1)$) belongs to an essential arc $\tilde{\alpha}:[0,1] \rightarrow \partial D_{0}$, the following properties hold. Denote by $\gamma$ the deck transformation such that $D_{0} \cap \gamma(D_{0})= \tilde{\alpha}$. The point $\tilde{\delta}_{i}(0)$ (respectively $\tilde{\delta}_{i}(1)$) belongs to the same essential arc contained in $\partial D_{0}$ as the point $\tilde{\delta}_{i,0}(0)$ (respectively $\tilde{\delta}_{i,0}(1)$). Moreover, the family consisting of the points $(\gamma \tilde{\delta}_{k,0} \cap \tilde{\alpha})_{k}$, where $k \leq i$ varies over the indices such that $\gamma \tilde{\delta}_{k,0} \cap \tilde{\alpha} \neq \emptyset$, and the point $\tilde{\delta}_{i,0}(0)$ (respectively $\tilde{\delta}_{i,0}(1)$)  are in the same order on $\tilde{\alpha}$ as the family consisting of the points $(\gamma \tilde{\delta}_{k} \cap \tilde{\alpha})_{k}$ and the point $\tilde{\delta}_{i}(0)$ (respectively $\tilde{\delta}_{i}(1)$).
\item If the point $\tilde{\delta}_{i,0}(0)$  (respectively $\tilde{\delta}_{i,0}(1)$) does not belong to one of these essential arcs, then the point $\tilde{\delta}_{i}(0)$ (respectively $\tilde{\delta}_{i}(1)$) belongs to the same component of $\varphi^{-1}([0,1] \times \left\{ 0 \right\}) \cap \partial \tilde{S}$ (respectively $\varphi^{-1}([0,1] \times \left\{ 1 \right\}) \cap \partial \tilde{S}$) as the point $\tilde{\delta}_{i,0}(0)$ (respectively $\tilde{\delta}_{i,0}(1)$).
\item  $\tilde{\delta}_{i}((0,1)) \subset D_{0}- \partial D_{0}$.
\item If $i>1$, the curve $\tilde{\delta}_{i}$ is strictly on the right of the curve $\tilde{\delta}_{i-1}$ in $D_{0}$.
\item For any $1 \leq j \leq L$, the curve $\tilde{\delta}_{i}$ meets the curve $\tilde{\beta}_{j}$ in only one point.
\item For any $-2^{L-i}\leq k,k' \leq 2^{L-i}$ and any $0 \leq j  \leq i$, the curves $\tilde{f}^{k}(\tilde{\delta}_{j})$ and $\tilde{f}^{k'}(\tilde{\delta}_{i})$ are disjoint.
\item For any $-2^{L-i} \leq k,k' \leq 2^{L-i}$ and any essential arc $\tilde{\alpha}$ contained in $\partial D_{0}$ which does not meet the curve $\tilde{\delta}_{i,0}$, the curves $\tilde{f}^{k}(\tilde{\delta}_{i})$ and $\tilde{f}^{k'}(\tilde{\alpha})$ are disjoint.
\item Consider any essential arc $\tilde{\alpha}:[0,1] \rightarrow \tilde{S}$ contained in $\partial D_{0}$ such that $\tilde{\alpha} \cap \tilde{\delta}_{i} \neq \emptyset$. Denote by $\gamma$ the deck transformation such that $D_{0} \cap \gamma D_{0}= \tilde{\alpha}$. Let $\left\{\tilde{\alpha}(t_{l}), 1 \leq l \leq r \right\}$, be the set of points of $\tilde{\alpha}$ which belong to one of the curves $\tilde{\delta}_{j}$ or $\gamma(\tilde{\delta}_{j'})$, where $1 \leq j,j' \leq i$ and $t_{1} < t_{2}< \ldots <t_{r}$. Let $t_{0}=0$ and $t_{r+1}=1$. Finally, let $\tilde{\alpha}(t_{l(i)})$ be the point $\tilde{\alpha} \cap \tilde{\delta}_{i}$. Then, for any $-2^{L-i} \leq k,k' \leq 2^{L-i}$, the following properties are satisfied.
\begin{itemize}
\item The image under $\tilde{f}^{k}$ of the arc $\tilde{\alpha}([0,t_{l(i)}])$ does not meet any of the following arcs: $\tilde{f}^{k'}(\tilde{\alpha}([t_{l(i)+1},1]))$, the curves of the form $\tilde{f}^{k'}(\tilde{\delta}_{j})$ if $j \leq i$ and $\tilde{\delta}_{j} \cap \tilde{\alpha}([0,t_{l(i)}])= \emptyset$ and the curves of the form $\tilde{f}^{k'}(\gamma(\tilde{\delta}_{j}))$ if $j \leq i$ and $\gamma(\tilde{\delta}_{j}) \cap \tilde{\alpha}([0,t_{l(i)}])= \emptyset$. Likewise, the image under $\tilde{f}^{k}$ of the arc $\tilde{\alpha}([t_{l(i)},1])$ does not meet any of the following arcs: $\tilde{f}^{k'}(\tilde{\alpha}([0,t_{l(i)-1}]))$, the curves of the form $\tilde{f}^{k'}(\tilde{\delta}_{j})$ if $j \leq i$ and $\tilde{\delta}_{j} \cap \tilde{\alpha}([t_{l(i)},1])= \emptyset$ and the curves of the form $\tilde{f}^{k'}(\gamma(\tilde{\delta}_{j}))$ if $j \leq i$ and $\gamma(\tilde{\delta}_{j}) \cap \tilde{\alpha}([t_{l(i)},1])= \emptyset$. 
\item The image under $\tilde{f}^{k}$ of the arc $\tilde{\delta}_{i}$ does not meet any of the following arcs:  $\tilde{f}^{k'}(\tilde{\alpha}([0,t_{l(i)-1}]))$ if $l(i)>0$, $\tilde{f}^{k'}(\tilde{\alpha}([t_{l(i)+1},1]))$ if $l(i)<r+1$ and the curves of the form $\tilde{f}^{k'}(\gamma(\tilde{\delta}_{j}))$ if $\tilde{\delta}_{i} \cap \gamma(\tilde{\delta}_{j})= \emptyset$, where $0 \leq j\leq i$. 
\end{itemize}
\end{enumerate}

Before completing the induction, let us why the curves $\tilde{\beta}_{i}$ and $\tilde{\delta}_{i}$ satisfy the required properties. From the properties satisfied by the curves $\tilde{\delta}_{i}$ and the curves $\tilde{\beta}_{i}$, we deduce that there exists a homeomorphism $\tilde{h}:D_{0} \rightarrow D_{0}$ with the following properties.
\begin{enumerate}
\item For any $1 \leq i \leq L$, $\tilde{h}(\tilde{\delta}_{i})=\tilde{\delta}_{i,0}$ and $\tilde{h}(\tilde{\beta}_{i})=\tilde{\beta}_{i,0}$.
\item The homeomorphism $\tilde{h}$ preserves any essential arc contained in $\partial D_{0}$.
\item For any two essential arcs $\tilde{\alpha}$ and $\tilde{\alpha}'$ contained in $\partial D_{0}$, if there exists a deck transformation $\gamma$ such that $\gamma(\tilde{\alpha}) = \tilde{\alpha}'$, then $\gamma \tilde{h}_{| \tilde{\alpha}}=\tilde{h}\gamma_{|\tilde{\alpha}}$. 
\end{enumerate}
The second and the third conditions above imply that the homeomorphism $\tilde{h}$ can be extended on $\tilde{S}$ as the lift of some homeomorphism $h$ in $\mathrm{Homeo}_{0}(S)$.
To construct such a homeomorphism, first construct it on the union of $\partial D_{0}$ with the curves $\tilde{\delta}_{i}$ and $\tilde{\beta}_{i}$. Then extend this homeomorphism to the connected components of the complement of this set in $D_{0}$ by using the Schönflies theorem.

Now, let us build the curves $\tilde{\delta}_{i}$ by induction. Fix an index $1 \leq i \leq L$ and suppose that we have constructed arcs $\tilde{\delta}_{1}, \tilde{\delta}_{2},\ldots, \tilde{\delta}_{i-1}$ with the above properties (this condition is empty in the case $i=1$). Denote by $A_{i}$ the set of essential arcs contained in $\partial D_{0}$ which lie strictly on the right of the curve $\tilde{\delta}_{i,0}$ in $D_{0}$ and by $B_{i}$ the set of essential arcs contained in $\partial D_{0}$ which lie strictly on the left of the curve $\tilde{\delta}_{i,0}$ in $D_{0}$. We distinguish three cases.\\
\textbf{First case:} The points $\tilde{\delta}_{i,0}(0)$ and $\tilde{\delta}_{i,0}(1)$ do not belong to an essential arc contained in $\partial D_{0}$.\\
\textbf{Second case:} The point $\tilde{\delta}_{i,0}(0)$ belongs to an essential arc $\tilde{\alpha}$ contained in $\partial D_{0}$ and the point $\tilde{\delta}_{i,0}(1)$ belongs to an essential arc $\tilde{\alpha}'$ contained in $\partial D_{0}$.\\
\textbf{Third case:} One of the points among $\tilde{\delta}_{i,0}(0)$ and $\tilde{\delta}_{i,0}(1)$ belongs to an essential arc contained in $\partial D_{0}$ and the other one does not.\\
We construct the curve $\tilde{\delta}_{i}$ in the first two cases and we leave the construction in the third case to the reader.

Let us look at the first case. Notice that the following properties hold.
\begin{enumerate}
\item For any $-2^{L-i+1}\leq k,k' \leq 2^{L-i+1}$, any essential arc $\tilde{\alpha}$ in $A_{i}$ and any essential arc $\tilde{\alpha}'$ in $B_{i}$, the curves $\tilde{f}^{k}(\tilde{\alpha})$ and $\tilde{f}^{k'}(\tilde{\alpha}')$ are disjoint.
\item For any $-2^{L-i+1} \leq k,k' \leq 2^{L-i+1}$ and any essential arc $\tilde{\alpha}$ in $A_{i}$, the curves $\tilde{f}^{k}(\tilde{\alpha})$ and  $\tilde{f}^{k'}(\tilde{\delta}_{i-1})$ are disjoint.
\item For any $-2^{L-i+1} \leq k \leq 2^{L-i+1}$, the arc $\tilde{f}^{k}(\delta_{i-1})$ is disjoint from the set $\varphi^{-1}(\left\{ 1 \right\} \times [0,1])$: by construction of the chart $\varphi$, the endpoints of these curves do not belong to the connected component of $\partial \tilde{S}$ which contains $\varphi^{-1}(\left\{ 1 \right\} \times [0,1])$.
\item For any $-2^{L-i+1}\leq k \leq 2^{L-i+1}$ and any essential arc $\tilde{\alpha}$ in $A_{i}$, the curve $\tilde{f}^{k}(\tilde{\alpha})$ is disjoint from the arcs of the form $\tilde{\beta}_{j}$, with $1 \leq j \leq L$.
\end{enumerate}
Hence there exists an essential arc $\tilde{\delta}_{i}$ with the following properties:
\begin{itemize}
\item The point $\tilde{\delta}_{i}(0)$ (respectively $\tilde{\delta}_{i}(1)$) belongs to the same component of $\partial \tilde{S} \cap D_{0}$ as the point $\tilde{\delta}_{i,0}(0)$ (respectively $\tilde{\delta}_{i,0}(1)$) and does not meet any essential arc contained in $\partial D_{0}$.
\item $\tilde{\delta}_{i}((0,1)) \subset D_{0}-\partial D_{0}$.
\item The curve $\tilde{\delta}_{i}$ lies strictly on the right of the curves of the form $\tilde{f}^{k}(\tilde{\delta}_{i-1})$, with $-2^{L-i+1}\leq k \leq 2^{L-i+1}$, in $D_{0}$.
\item For any $1 \leq j \leq L$, the curve $\tilde{\delta}_{i}$ meets the curve $\tilde{\beta}_{j}$ in only one point.
\item The curves of the form $\tilde{f}^{k}(\tilde{\alpha})$, where $\tilde{\alpha}$ belongs to $A_{i}$ and $-2^{L-i+1}\leq k \leq 2^{L-i+1}$, lie strictly on the right of the curve $\tilde{\delta}_{i}$ and the curves of the form $\tilde{f}^{k}(\tilde{\alpha})$, where $\tilde{\alpha}$ belongs to $B_{i}$ and $-2^{L-i+1}\leq k \leq 2^{L-i+1}$, lie strictly on the left of the curve $\tilde{\delta}_{i}$.
\end{itemize}
Of course, the curves of the form $\tilde{f}^{k}(\tilde{\delta}_{j})$, with $j<i-1$ and $-2^{L-i+1}\leq k \leq 2^{L-i+1}$, lie strictly on the left of the curve $\tilde{f}^{k}(\tilde{\delta}_{i-1})$ and hence are disjoint from the curve $\tilde{\delta}_{i}$. The curve $\tilde{\delta}_{i}$ satisfies the required properties.

The second case is subdivided into three subcases.\\
\textbf{First subcase:} The arcs of the form $\gamma(\tilde{\delta}_{j,0})$, where $\gamma$ is a nontrivial deck transformation and $1 \leq j \leq i-1$, meet neither the arc $\tilde{\alpha}$ nor the arc $\tilde{\alpha}'$.\\
\textbf{Second subcase:} The arcs $\tilde{\alpha}$ and $\tilde{\alpha}'$ both meet an arc of the form $\gamma(\tilde{\delta}_{j,0})$, where $\gamma$ is a nontrivial deck transformation and $1 \leq j \leq i-1$.\\
\textbf{Third subcase} One of the arcs $\tilde{\alpha}$ and $\tilde{\alpha}'$ meets an arc of the form $\gamma(\tilde{\delta}_{j,0})$, where $\gamma$ is a nontrivial deck transformation and $1 \leq j \leq i-1$ and the other does not.\\
We construct the arc $\tilde{\delta}_{i}$ only in the first two subcases and leave the construction to the reader in the third one.
Changing the orientation if necessary, we can suppose that the arcs $\tilde{\alpha}:[0,1] \rightarrow \partial D_{0}$ and $\tilde{\alpha}': [0,1] \rightarrow \partial D_{0}$ are oriented in such a way that the points $\tilde{\alpha}(1)$ and $\tilde{\alpha}'(1)$ lie on the right of the curve $\tilde{\delta}_{i,0}$ in $D_{0}$.

Let us study the first subcase. Let $\tau$ be the parameter in $[0,1]$ defined by $\tau=0$ if $\delta_{i-1} \cap \tilde{\alpha} = \emptyset$ and $\left\{ \tilde{\alpha}(\tau) \right\}= \tilde{\delta}_{i-1} \cap \tilde{\alpha}$ otherwise. Let $\tau'$ be the parameter in $[0,1]$ defined by $\tau'=0$ if $\delta_{i-1} \cap \tilde{\alpha}' = \emptyset$ and $\left\{ \tilde{\alpha}'(\tau') \right\}= \tilde{\delta}_{i-1} \cap \tilde{\alpha}'$ otherwise. 
Then take an arc $\tilde{\delta}_{i}$ with the following properties.
\begin{enumerate}
\item The point $\tilde{\delta}_{i}(0)$ belongs to the arc $\tilde{\alpha}$ and the point $\tilde{\delta}_{i}(1)$ belongs to the arc $\tilde{\alpha}'$.
\item $\tilde{\delta}_{i}((0,1)) \subset D_{0}- \partial D_{0}$.
\item The curve $\tilde{\delta}_{i}$ meets each of the curves $\tilde{\beta}_{j}$, with $1 \leq j \leq L$, in only one point.
\item The compact sets of the form $\tilde{f}^{k}(\tilde{\delta}_{i-1} \cup \tilde{\alpha}([0,\tau]) \cup \tilde{\alpha}'([0,\tau']))\cap D_{0}$ with $-2^{L-i+1}\leq k \leq 2^{L-i+1}$ lie strictly on the left of the arc $\tilde{\delta}_{i}$ in $D_{0}$.
\item The curves of the form $\tilde{f}^{k}(\tilde{\alpha}'')$, where $\tilde{\alpha}''$ belongs to $A_{i}$ and $-2^{L-i+1}\leq k \leq 2^{L-i+1}$, lie strictly on the right of the curve $\tilde{\delta}_{i}$ in $D_{0}$ and the curves of the form $\tilde{f}^{k}(\tilde{\alpha}'')$, where $\tilde{\alpha}''$ belongs to $B_{i}$ and $-2^{L-i+1}\leq k \leq 2^{L-i+1}$, lie strictly on the left of the curve $\tilde{\delta}_{i}$ in $D_{0}$.
\end{enumerate}
The arc $\tilde{\delta}_{i}$ satisfies the required properties.

Now, we look at the second subcase (see Figure \ref{secondsubcase} for an illustration of the notation). Denote by $\gamma$ the deck transformation such that $\gamma(D_{0}) \cap D_{0}= \tilde{\alpha}$ and by $\gamma'$ the deck transformation such that $\gamma'(D_{0}) \cap D_{0}= \tilde{\alpha}'$. Let $\tilde{\alpha}(t_{0})=\tilde{\delta}_{i,0}(0)$ and $\tilde{\alpha}'(t'_{0})=\tilde{\delta}_{i,0}(1)$. Denote by $t_{-,0}$ (respectively $t'_{-,0}$) the supremum of the real numbers $t< t_{0}$ (respectively $t<t'_{0}$) such that $\tilde{\alpha}(t)$ (respectively $\tilde{\alpha}'(t)$) meets an arc of the form $\gamma \tilde{\delta}_{j,0}$ (respectively $\gamma' \tilde{\delta}_{j,0}$) with $j<i$. Take $t_{-,0}=0$ (respect. $t'_{-,0}=0$) if there is no such real number. Likewise, denote by $t_{+,0}$ (respectively $t'_{+,0}$) the infimum of the real numbers $t> t_{0}$ (respectively $t>t'_{0}$) such that $\tilde{\alpha}(t)$ (respectively $\tilde{\alpha}'(t)$) meets an arc of the form $\gamma \tilde{\delta}_{j,0}$ (respectively $\gamma' \tilde{\delta}_{j,0}$) with $j<i$. Take $t_{+,0}=1$ (respectively $t'_{+,0}=1$) if there is no such real number. Denote by $j_{-}$, $j_{+}$, $j'_{-}$ and $j'_{+}$ the indices such that, respectively, the point $\tilde{\alpha}(t_{-,0})$ belongs to the arc $\gamma( \tilde{\delta}_{j_{-},0})$, the point $\tilde{\alpha}'(t'_{-,0})$ belongs to the arc $\gamma'(\tilde{\delta}_{j'_{-},0})$, the point $\tilde{\alpha}(t_{+,0})$ belongs to the arc $\gamma(\tilde{\delta}_{j_{+},0})$ and the point $\tilde{\alpha}'(t'_{+,0})$ belongs to the arc $\gamma'(\tilde{\delta}_{j'_{+},0})$. Finally, denote by $t_{-}$, $t_{+}$, $t'_{-}$ and $t'_{+}$ the real numbers in $[0,1]$ such that, respectively, the point $\tilde{\alpha}(t_{-})$ belongs to the arc $\gamma (\tilde{\delta}_{j_{-}})$, the point $\tilde{\alpha}(t_{+})$ belongs to the arc $\gamma (\tilde{\delta}_{j_{+}})$, the point $\tilde{\alpha}'(t_{-})$ belongs to the arc $\gamma' (\tilde{\delta}_{j'_{-}})$, the point $\tilde{\alpha}'(t_{+})$ belongs to the arc $\gamma' (\tilde{\delta}_{j'_{+}})$. Finally, as in the first subcase, let $\tau$ (respectively $\tau'$) be the parameter in $[0,1]$ defined by $\tau=0$ (respectively $\tau'=0$) if $\tilde{\delta}_{i-1} \cap \tilde{\alpha} = \emptyset$ (respectively $\tilde{\delta}_{i-1} \cap \tilde{\alpha}' = \emptyset$) and $\left\{ \tilde{\alpha}(\tau) \right\} = \tilde{\delta}_{i-1} \cap \tilde{\alpha}$ (respectively $\left\{ \tilde{\alpha}'(\tau') \right\} = \tilde{\delta}_{i-1} \cap \tilde{\alpha}'$) otherwise.

\begin{figure}[ht]
\begin{center}
\includegraphics{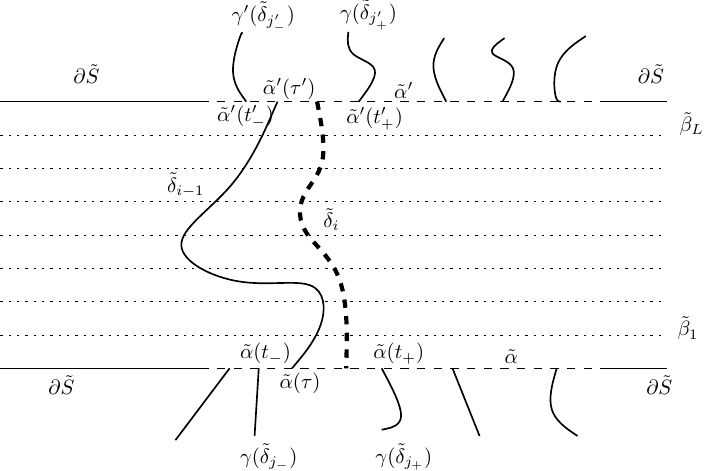}
\end{center}
\caption{Notation in the second subcase}
\label{secondsubcase}
\end{figure}

There exists an arc $\tilde{\delta}_{i}:[0,1] \rightarrow D_{0}$ with the following properties.
\begin{enumerate}
\item The point $\tilde{\delta}_{j}(0)$ belongs to the arc $\tilde{\alpha}$ and the point $\tilde{\delta}_{j}(1)$ to the arc $\tilde{\alpha}'$.
\item $\tilde{\delta}_{i}((0,1)) \subset D_{0}- \partial D_{0}$.
\item The curve $\tilde{\delta}_{i}$ meets each of the curves $\tilde{\beta}_{j}$, with $1 \leq j \leq L$, in only one point.
\item The compact sets of the form $\tilde{f}^{k}(\tilde{\delta}_{i-1})\cap D_{0}$, $\tilde{f}^{k}(\tilde{\alpha}([0, \max(\tau, t_{-})]))\cap D_{0}$, $\tilde{f}^{k}(\tilde{\alpha}'([0, \max(\tau',t'_{-})])) \cap D_{0}$, $\tilde{f}^{k} \gamma(\tilde{\delta}_{j_{-}})\cap D_{0}$ and $\tilde{f}^{k} \gamma(\tilde{\delta}_{j'_{-}})\cap D_{0}$ lie strictly on the left of $\tilde{\delta}_{i}$ in $D_{0}$.
\item The compact sets of the form $\tilde{f}^{k}(\tilde{\alpha}([t_{+},1]))\cap D_{0}$, $\tilde{f}^{k}(\tilde{\alpha}'([t'_{+},1])) \cap D_{0}$, $\gamma(\tilde{\delta}_{j_{+}})\cap D_{0}$ and $\tilde{f}^{k}\gamma(\tilde{\delta}_{j'_{+}})\cap D_{0}$ lie strictly on the right of $\delta_{i}$ in $D_{0}$.
\item The curves of the form $\tilde{f}^{k}(\tilde{\alpha}'')$, where $\tilde{\alpha}''$ belongs to $A_{i}$ and $-2^{L-i+1}\leq k \leq 2^{L-i+1}$, lie strictly on the right of the curve $\tilde{\delta}_{i}$ in $D_{0}$ and the curves of the form $\tilde{f}^{k}(\tilde{\alpha}'')$, where $\tilde{\alpha}''$ belongs to $B_{i}$ and $-2^{L-i+1}\leq k \leq 2^{L-i+1}$, lie strictly on the left of the curve $\tilde{\delta}_{i}$ in $D_{0}$.
\end{enumerate}
Such an arc satisfies the required conditions (note that, by construction of the chart $\varphi$, $\gamma \tilde{\alpha}' \neq \tilde{\alpha}$ and $\gamma' \tilde{\alpha} \neq \tilde{\alpha}'$ hence the curve $\gamma \tilde{\delta}_{i}$ does not meet the arc $\tilde{\alpha}$ and the curve $\gamma' \tilde{\delta}_{i}$ does not meet the arc $\tilde{\alpha}$).
The induction is complete.
\end{proof}

Now we turn to the proof of Proposition \ref{disjointcurvesboundary}.
\begin{proof}[Proof of Proposition \ref{disjointcurvesboundary}]
Denote by $\mathcal{G}$ the set of deck transformations $\gamma$ of $\tilde{S}$ such that
\begin{enumerate}
\item The fundamental domains $\gamma(D_{0})$ and $D_{0}$ have an essential arc in common.
\item The fundamental domain $D_{0}$ lies on the left of this essential arc.
\end{enumerate}
The group $\pi_{1}(S)$ is the free group on the elements of $\mathcal{G}$.

Let $l= 1-\chi(S)$. We prove by induction on $0 \leq b \leq l$ that, for any $N \geq 1$, the following property $\mathcal{P}(N)$ holds. There exists a family $(\alpha_{i})_{1 \leq i \leq b}$ of pairwise disjoint and pairwise non-homotopic essential arcs of $S$ and a homeomorphism $g_{b}$ in $\mathrm{Homeo}_{0}(S)$ with the following properties.
\begin{enumerate}
\item For any index $1 \leq i \leq b$, $g_{b}(\alpha_{i})= \alpha_{i,0}$.
\item Let $D_{b}=\tilde{g}_{b}^{-1}(D_{0})$ and denote by $\tilde{\alpha}_{i}: [0,1] \rightarrow \tilde{S}$ the lift of $\alpha_{i}$ such that $\tilde{\alpha}_{i}([0,1]) \subset D_{b}$ and the fundamental domain $D_{b}$ lies on the left of $\tilde{\alpha}_{i}$. Then, for any deck transformation $\gamma \in \pi_{1}(S)$, any indices $1 \leq i \neq j \leq b$ and any $\left| k \right| \leq N (l_{\mathcal{G}}(\gamma)+1)$, $\gamma \tilde{f}^{k}(\tilde{\alpha}_{i}) \cap \tilde{\alpha}_{j} = \emptyset$.
\item For any non-trivial automorphism $\gamma \in \pi_{1}(S)$, any $1 \leq i \leq b$ and any $\left| k \right| \leq Nl_{\mathcal{G}}(\gamma)$, $\gamma \tilde{f}^{k}(\tilde{\alpha}_{i}) \cap \tilde{\alpha}_{i} = \emptyset$.
\end{enumerate}

Notice that, if the above properties hold for $b=l$, Proposition \ref{disjointcurvesboundary} is proved.

For $b=0$, there is nothing to prove.

Suppose that the above property holds for $b < l$ and let us prove this for $b+1$. Fix $N' \geq 1$ and let $N=2N'$. Consider a family $(\alpha_{i})_{1 \leq i \leq b}$ of pairwise disjoint and pairwise non-homotopic essential arcs and a homeomorphism $g_{b}$ which satisfy $\mathcal{P}(3N)=\mathcal{P}(6N')$. Let $\tilde{\beta}_{0}=\tilde{g}_{b}^{-1}(\tilde{\alpha}_{b+1,0})$, where $\tilde{\alpha}_{b+1,0}$ is the lift of the arc $\alpha_{b+1,0}$ such that the fundamental domain $D_{0}$ lies on the left of the arc $\tilde{\alpha}_{b+1,0}$.

For any reduced words $w$ and $w'$ in elements of $\mathcal{G}$, denote by $w \pi_{1}(S)w'$ the set of automorphisms in $\pi_{1}(S)$ whose reduced representative starts with the word $w$ and ends with the word $w'$. Denote by $a$ the element of the generating set $\mathcal{G}$ such that $D_{b} \cap a(D_{b})= \tilde{\beta}_{0}$. It is also the element of the generating set $\mathcal{G}$ such that $D_{0} \cap a(D_{0})= \tilde{\alpha}_{b+1,0}$, as $\tilde{g}_{b} a= a \tilde{g}_{b}$. 

We will use the following fact repeatedly.

\noindent \underline{Fact:} For any automorphism $\gamma \in \pi_{1}(S)$, the fundamental domain $\gamma(D_{b})$ lies on the right of $\tilde{\beta}_{0}$ if and only if $\gamma \in a \pi_{1}(S)$. 

Let $\Gamma$ be the set of deck transformations whose reduce representative does not begin with the letter $a$ and does not end with the letter $a^{-1}$. The proof relies on the following lemma.

\begin{lemma} \label{goodcurve}
There exists an essential arc $\tilde{\beta}_{2}:[0,1] \rightarrow \tilde{S}$ which satisfies the following properties.
\begin{enumerate}
\item The point $\tilde{\beta}_{2}(0)$ belongs to the same connected component of $\partial \tilde{S}$ as $\tilde{\beta}_{0}(0)$ and the point $\tilde{\beta}_{2}(1)$ belongs to the same connected component of $\partial \tilde{S}$ as $\tilde{\beta}_{0}(1)$.
\item For any deck transformation $\gamma$ in $a \pi_{1}(S)$ and any integer $\left| k \right| \leq N (l_{\mathcal{G}}(\gamma)+1)$, the curves of the form $\gamma \tilde{f}^{k}(\tilde{\alpha}_{i})$ lie strictly on the right of the curve $\tilde{\beta}_{2}$.
\item For any deck transformation $\gamma$ in $a \pi_{1}(S)$ and any integer $\left| k \right| \leq N l_{\mathcal{G}}(\gamma)$, the curves of the form $\gamma \tilde{f}^{k}(\tilde{\beta}_{2})$ lie strictly on the right of the curve $\tilde{\beta}_{2}$.
\item For any deck transformation $\gamma$ in $\Gamma$ different from the identity and for any integer $\left| k \right| \leq N l_{\mathcal{G}}(\gamma)$, the curve $\gamma \tilde{f}^{k}(\tilde{\beta}_{2})$ lies strictly on the left of the curve $\tilde{\beta}_{2}$.
\item For any deck transformation $\gamma$ in $\Gamma \cup \pi_{1}(S)a^{-1}-a \pi_{1}(S)a^{-1}$ and for any integer $\left| k \right| \leq N (l_{\mathcal{G}}(\gamma)+1)$, the curves of the form $\gamma \tilde{f}^{k}(\tilde{\alpha}_{i})$ lie strictly on the left of the curve $\tilde{\beta}_{2}$.
\end{enumerate}
\end{lemma}

This lemma is proved below. We now explain how to complete the induction. Let $\tilde{\alpha}_{b+1}= \tilde{\beta}_{2}$.

By construction of the curve $\tilde{\alpha}_{b+1}$, for any non-trivial deck transformation $\gamma \in a\pi_{1}(S) \cup \Gamma$ and any $\left| k \right| \leq N' l_{\mathcal{G}}(\gamma)$, we have $\gamma \tilde{f}^{k}(\tilde{\alpha}_{b+1}) \cap \tilde{\alpha}_{b+1} = \emptyset$. Hence, for any element $\gamma$ of $\pi_{1}(S)a^{-1}$ and any $\left|k \right| \leq N' l_{\mathcal{G}}(\gamma)$, $\gamma \tilde{f}^{k}(\tilde{\alpha}_{b+1}) \cap \tilde{\alpha}_{b+1}= \gamma \tilde{f}^{k}( \tilde{\alpha}_{b+1} \cap \gamma^{-1}\tilde{f}^{-k}( \tilde{\alpha}_{b+1})) = \emptyset$.

 As $\pi_{1}(S)=\pi_{1}(S)a^{-1} \cup a \pi_{1}(S) \cup \Gamma$, we have proved that for any element $\gamma$ of $\pi_{1}(S)$ and any $\left| k \right| \leq N' l_{\mathcal{G}}(\gamma)$, $\gamma \tilde{f}^{k}(\tilde{\alpha}_{b+1}) \cap \tilde{\alpha}_{b+1}=\emptyset$.

Moreover, Lemma \ref{goodcurve} implies that, for any deck transformation $\gamma \in \pi_{1}(S)$ and any indices $1 \leq i \neq j \leq b+1$, if $\left| k \right| \leq N' (l_{\mathcal{G}}(\gamma)+1)$, then $\gamma \tilde{f}^{k}(\tilde{\alpha}_{i}) \cap \tilde{\alpha}_{j} = \emptyset$.

Observe that the projection $\alpha_{b+1}$ on $S$ of the arc $\tilde{\alpha}_{b+1}$ is an essential arc. Now we construct the homeomorphism $g_{b+1}$. Notice that the arc $\alpha_{b+1}$ is homotopic to the arc $\beta_{0}$ relative to $\partial S \cup \cup_{1 \leq i \leq b}\alpha_{i}$. Hence, by the main theorem of the article \cite{Eps} by Epstein, there exists a homeomorphism $g'_{b+1}$ in $\mathrm{Homeo}_{0}(S)$ which pointwise fixes the curves $\alpha_{i}$ for $1 \leq i \leq b$ and which sends the curve $\beta_{0}$ to the curve $\alpha_{b+1}$. Then take $g_{b+1}=g'_{b+1} g_{b}$.
\end{proof}

It remains to prove Lemma \ref{goodcurve}. We need the following lemma (see Figure \ref{infcurves}).

\begin{lemma} \label{inf}
Let $\tilde{\alpha}$ be an essential arc of $\tilde{S}$. Denote by $\mathcal{C}$ the set of boundary components of $\tilde{S}$ which lie strictly on the left of $\tilde{\alpha}$. Let $(\tilde{\alpha}_{i})_{1 \leq i \leq n}$ be a finite family of essential arcs of $\tilde{S}$ such that, for any $i$, any component in $\mathcal{C}$ lies strictly on the left of $\tilde{\alpha}_{i}$.

Then there exists a unique essential arc $\inf_{\tilde{\alpha}}((\tilde{\alpha}_{i})_{1 \leq i \leq n})$ with the following properties.
\begin{enumerate}
\item Any essential arc of $\tilde{S}$ which lies (strictly) on the left $\tilde{\alpha}$ and of the $\tilde{\alpha}_{i}$'s lies (strictly) on the left of the arc $\inf_{\tilde{\alpha}}((\tilde{\alpha}_{i})_{1 \leq i \leq n})$.
\item The arc $\inf_{\tilde{\alpha}}((\tilde{\alpha}_{i})_{1 \leq i \leq n})$ lies on the left of the arc $\tilde{\alpha}$ and of the arcs $\tilde{\alpha}_{i}$.
\end{enumerate}
Moreover, any point of the essential arc $\inf_{\tilde{\alpha}}((\tilde{\alpha}_{i})_{1 \leq i \leq n})$ belongs to either the arc $\tilde{\alpha}$ or one of the arcs $\tilde{\alpha}_{i}$.
\end{lemma}

\begin{figure}[ht]
\begin{center}
\includegraphics{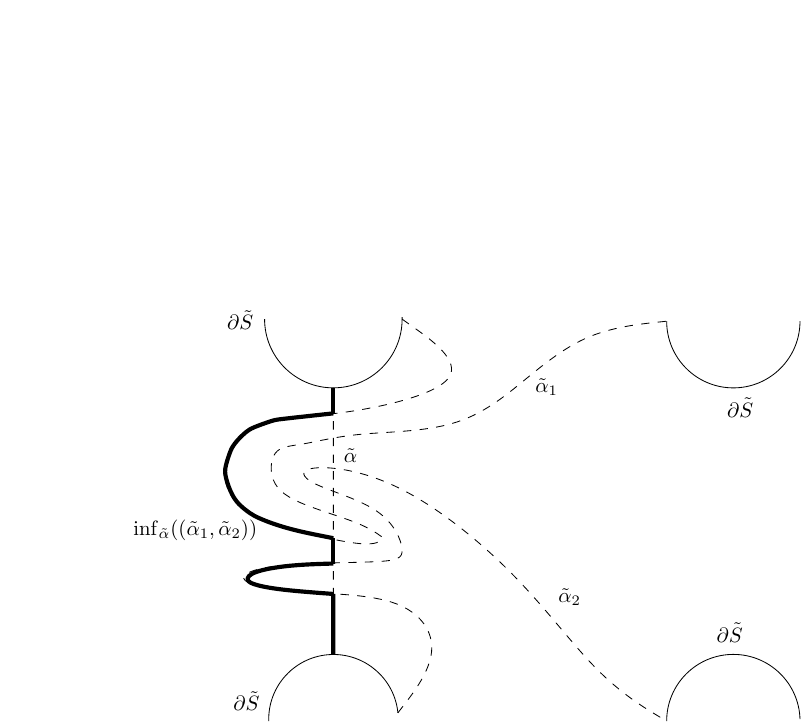}
\end{center}
\caption{Illustration of Lemma \ref{inf}}
\label{infcurves}
\end{figure}

Of course, in the above lemma, if we replace the words "left" by "right", we can define a curve $\sup_{\tilde{\alpha}}((\tilde{\alpha}_{i})_{1 \leq i \leq n})$.

\begin{proof}
We use the following lemma by Kerekjarto (see \cite{LCY} p. 246 for a proof).

\begin{lemma} \label{Kerekjarto}
Let $U_{1}, U_{2}, \ldots, U_{k}$ be Jordan domains of the sphere, that is, connected components of the complement of a Jordan curve. Then each connected component of $U_{1}\cap U_{2} \cap \ldots \cap U_{k}$ is a Jordan domain.
\end{lemma}

We see the Poincaré disk as the unit disk in the plane which is seen as the Riemann sphere minus the point at infinity. As the universal cover of the double of the surface $S$, endowed with a hyperbolic metric, is the Poincaré disk, the surface $\tilde{S}$ is naturally a subset of the Poincaré disk and the interior of $\tilde{S}$ is a Jordan domain $J_{0}$ -its frontier is locally connected without cut points. Denote by $U$ the connected component of $int(\tilde{S})-\tilde{\alpha}$ which lies on the left of $\tilde{\alpha}$. For any $1 \leq i \leq n$, denote by $U_{i}$ the connected component of $int(\tilde{S})-\tilde{\alpha}_{i}$ which lies on the left of $\tilde{\alpha}_{i}$.

As any component in $\mathcal{C}$ is strictly on the left of $\tilde{\alpha}$ and of the $\tilde{\alpha}_{i}$'s, any point in $int(\tilde{S})$ close to such a component belongs to $U \cap \bigcap_{i=1}^{n} U_{i}$. Moreover, the components in $\mathcal{C}$ are contained in the closure of the same connected component of $U \cap \bigcap_{i=1}^{n} U_{i}$: they belong to the same connected component of $\partial J_{0} - (\bigcup_{i=1}^{n}\tilde{\alpha}_{i} \cup \tilde{\alpha})$. Let us call $J$ this connected component of $U \cap \bigcap_{i=1}^{n} U_{i}$. By Lemma \ref{Kerekjarto}, the set $J$ is a Jordan domain. Hence its frontier inside $int(\tilde{S})$ (or more precisely the closure in $\tilde{S}$ of this frontier) is an essential arc. If we orient properly this arc, it satisfies the two required properties. To prove the uniqueness part of this lemma, observe that the image of an arc which satisfies the two required properties is necessarily the closure in $\tilde{S}$ of the frontier in $int(\tilde{S})$ of the component $J$: the first condition implies that it is contained in the closure of $U \cap \bigcap_{i=1}^{n} U_{i}$. As such an arc touches a boundary component of $\partial \tilde{S}$, it must be contained in the closure of $J$. The second condition implies that this arc must be contained in the frontier of $J$. As the complement of $\partial \tilde{S}$ in this frontier is connected and the orientation of our arc is determined by the two conditions, the uniqueness part of the lemma is proved.
\end{proof}

\begin{proof}[Proof of Lemma \ref{goodcurve}]
\textbf{First step.} We prove first the following property. \emph{There exists an essential arc $\tilde{\beta}_{1}$ of $\tilde{S}$ such that
\begin{enumerate}
\item The point $\tilde{\beta}_{1}(0)$ lies on the same connected component of $\partial \tilde{S}$ as the point $\tilde{\beta}_{0}(0)$ and the point $\tilde{\beta}_{1}(1)$ lies on the same connected component of $\partial \tilde{S}$ as the point $\tilde{\beta}_{0}(1)$.
\item For any deck transformation $\gamma$ in $a \pi_{1}(S)$ and any integer $\left| k \right| \leq N (l_{\mathcal{G}}(\gamma)+1)$, the curves of the form $\gamma \tilde{f}^{k}(\tilde{\alpha}_{i})$ lie on the right of the curve $\tilde{\beta}_{1}$. 
\item For any deck transformation $\gamma$ in $a \pi_{1}(S)$ and any integer $\left| k \right| \leq N l_{\mathcal{G}}(\gamma)$, the curves of the form $\gamma \tilde{f}^{k}(\tilde{\beta}_{1})$ lie on the right of the curve $\tilde{\beta}_{1}$.
\end{enumerate}}

We start with a lemma where we use the non-spreading hypothesis.

\begin{lemma} \label{finitenumber}
In the set of essential arcs 
$$ \left\{ \gamma \tilde{f}^{k}(\tilde{\beta}_{0}),
\left\{
\begin{array}{l}
\left| k \right| \leq N(l_{\mathcal{G}}(\gamma)+1) \\
\gamma \in a \pi_{1}(S) 
\end{array}
\right.
\right\} \cup \left\{ \gamma(\tilde{f}^{k}(\tilde{\alpha}_{i})), 
\left\{
\begin{array}{l}
 \left| k \right| \leq N(l_{\mathcal{G}}(\gamma)+1) \\
1 \leq i \leq b, \gamma \in a \pi_{1}(S)
\end{array}
\right.
\right\},$$
only a finite number of arcs meet the curve $\tilde{\beta}_{0}$.
\end{lemma}

\begin{proof}
Suppose for a contradiction that there exists a sequence $(\gamma_{n})_{n \in \mathbb{N}}$ of non-trivial automorphisms in $\pi_{1}(S)$ with $l_{\mathcal{G}}(\gamma_{n}) \xrightarrow[n \to + \infty]{} +\infty $ and a sequence $(k_{n})_{n \in \mathbb{N}}$ of integers with $\left|k_{n} \right| \leq N(l_{\mathcal{G}}(\gamma_{n})+1)$ such that, for any $n$, there exists a curve $\tilde{\beta}$ among $\tilde{\beta}_{0}$ and the $\tilde{\alpha}_{i}$'s such that
$$ \gamma_{n}\tilde{f}^{k_{n}}(\tilde{\beta}) \cap \tilde{\beta}_{0} \neq \emptyset.$$
Then, for any $n$, $\tilde{f}^{k_{n}}(D_{b}) \cap \gamma_{n}^{-1}(D_{b}) \neq \emptyset$. As the homeomorphism $\tilde{f}$ has a fixed point in $D_{b}$ (otherwise we could build a nowhere vanishing vector field on the surface $S$, which is impossible),
$$\mathrm{diam}(\tilde{f}^{k_{n}}(D_{b})) \geq d(D_{b}, \gamma_{n}(D_{b})).$$
By the \v{S}varc--Milnor lemma (see \cite{Har} p.87), there exists constants $C,C' >0$ such that, for any $n$,
$$ \mathrm{diam}(\tilde{f}^{k_{n}}(D_{b})) \geq C l_{\mathcal{G}}(\gamma_{n})-C'.$$
Therefore
$$ \frac{\mathrm{diam}(\tilde{f}^{k_{n}}(D_{b}))}{k_{n}} \geq \frac{C l_{\mathcal{G}}(\gamma_{n})}{N (l_{\mathcal{G}}(\gamma_{n})+1)}- \frac{C'}{N(l_{\mathcal{G}}(\gamma_{n})+1)}.$$
the right hand side of this inequality has a positive limit as $n \rightarrow + \infty$. Moreover, the sequence $(|k_{n}|)_{n}$ has to tend to $+\infty$: otherwise, one of the sets of the form $\tilde{f}^{l}(D_{b})$, with $l$ in $\mathbb{Z}$, would have infinite diameter as it would cross infinitely many sets of the form $\gamma_{n}(D_{b})$. This contradicts the hypothesis $\lim_{n \to +\infty} \frac{\mathrm{diam}(\tilde{f}^{n}(D_{b}))}{n}=0$ (recall that this hypothesis is independent of the chosen fundamental domain).
\end{proof}

We now want to apply Lemma \ref{inf} to complete this first step. We have first to check that the family of essential arcs we will consider satisfies the hypothesis of this lemma. For any essential arc $c$ of $\tilde{S}$ we denote by $c^{0}$ the arc with the opposite orientation. We denote by $\mathcal{C}$ the set of connected components of $\partial \tilde{S}$ which lie on the left of $\tilde{\beta}_{0}$.

Recall first that the essential arcs among the lifts of the arc $\beta_{0}$ which lie strictly on the right of $\tilde{\beta}_{0}$ are the curves of the form $\gamma(\tilde{\beta}_{0})$, where $\gamma$ belongs to $a \pi_{1}(S)$. Therefore the arc $\tilde{\beta}_{0}$ lies strictly on the left of the curves of the form $\gamma(\tilde{\beta}_{0})$, with $\gamma \in a \pi_{1}(S)-a\pi_{1}(S)a^{-1}$, as the curve $\gamma^{-1}(\tilde{\beta}_{0})$ is strictly on the left of the curve $\tilde{\beta}_{0}$. Moreover, the arc $\tilde{\beta}_{0}$ lies strictly on the right of the curves of the form $\gamma(\tilde{\beta}_{0})$, with $\gamma \in a\pi_{1}(S) a^{-1}$. By this discussion, the components in $\mathcal{C}$ lie strictly on the left of the curves of the form $\gamma(\tilde{\beta}_{0})$, hence also of the form $\gamma \tilde{f}^{k}(\tilde{\beta}_{0})$, with $k \in \mathbb{Z}$ and $\gamma \in a \pi_{1}(S)-a\pi_{1}(S)a^{-1}$. They also lie strictly on the left of the curves of the form $\gamma \tilde{f}^{k}(\tilde{\beta}_{0}^{0})$, with $k \in \mathbb{Z}$ and $\gamma \in a \pi_{1}(S)a^{-1}$.

For any $1 \leq i \leq b$, denote by $A_{i}$ the subset of $a\pi_{1}(S)$ consisting of deck transformations $\gamma$ such that $D_{b}$ lies strictly on the left of the arc $\gamma(\tilde{\alpha}_{i})$. Denote by $A_{i}^{c}$ the complement of this set in $a \pi_{1}(S)$.
Consider the family $\mathcal{F}$ of essential arcs consisting of the arcs which meet the arc $\tilde{\beta}_{0}$ of one of the following forms.
\begin{enumerate}
\item $\gamma \tilde{f}^{k}(\tilde{\beta}_{0}), \gamma \in a \pi_{1}(S)-a\pi_{1}(S)a^{-1}, \left| k \right| \leq N l_{\mathcal{G}}(\gamma)$.
\item $\gamma \tilde{f}^{k}(\tilde{\beta}_{0}^{0}), \gamma \in a\pi_{1}(S)a^{-1}, \left| k \right| \leq N l_{\mathcal{G}}(\gamma)$.
\item $\gamma\tilde{f}^{k}(\tilde{\alpha_{i}}), 1 \leq i \leq b, \gamma \in A_{i}, \left| k \right| \leq N(l_{\mathcal{G}}(\gamma)+1)$
\item $\gamma \tilde{f}^{k}(\tilde{\alpha}_{i}^{0}), 1 \leq i \leq b, \gamma \in A_{i}^{c}, \left| k \right| \leq N(l_{\mathcal{G}}(\gamma)+1)$
\end{enumerate}

By Lemma \ref{finitenumber}, the family $\mathcal{F}$ is finite. Take $\tilde{\beta}_{1}=\inf_{\tilde{\beta}_{0}}( \mathcal{F}).$ Let us check that this curve satisfies the wanted properties.

By construction of the curve $\tilde{\beta}_{1}$, for any deck transformation $\gamma$ in $a \pi_{1}(S)$ and any integer $\left| k \right| \leq N (l_{\mathcal{G}}(\gamma)+1)$, the curves of the form $\gamma \tilde{f}^{k}(\tilde{\alpha}_{i})$ lie on the right of the curve $\tilde{\beta}_{1}$. 

Now, take a deck transformation $\gamma \in a \pi_{1}(S)- a \pi_{1}(S)a^{-1}$ and $\left| k \right| \leq N l_{\mathcal{G}}(\gamma)$. Observe that, by the uniqueness part of Lemma \ref{inf}, $\gamma\tilde{f}^{k}(\tilde{\beta}_{1})=\inf_{\gamma \tilde{f}^{k}(\tilde{\beta}_{0})}(\gamma\tilde{f}^{k}(\mathcal{F}))$, where $\gamma\tilde{f}^{k}(\mathcal{F})= \left\{ \gamma \tilde{f}^{k}(\tilde{\beta}), \tilde{\beta} \in \mathcal{F} \right\}$. It is easy to check that any curve in $\gamma\tilde{f}^{k}(\mathcal{F})$ is either equal to one of the curves in $\mathcal{F}$ or lies strictly on the right of $\tilde{\beta}_{0}$. Hence any of these curves lie on the right of $\tilde{\beta}_{1}$. As the curve $\tilde{f}(\tilde{\beta}_{0})$ also lies on the right of the curve $\tilde{\beta}_{1}$ by construction, we deduce that the curve $\gamma\tilde{f}^{k}(\tilde{\beta}_{1})$ lies on the right of the curve $\tilde{\beta}_{1}$, by Lemma \ref{inf}.

Finally, take a deck transformation $\gamma \in a \pi_{1}(S)a^{-1}$ and $\left| k \right| \leq N l_{\mathcal{G}}(\gamma)$. As the arc $\tilde{\beta}_{1}$ lies on the left of the arc $\tilde{\beta}_{0}$, the arc $\gamma \tilde{f}^{k}(\tilde{\beta}_{1})$ lies on the left of the arc $\gamma \tilde{f}^{k}(\tilde{\beta}_{0})$. As the arc $\tilde{\beta}_{1}$ lies on the right of the arc  $\gamma \tilde{f}^{k}(\tilde{\beta}_{0})$ by construction and as the arc $\gamma \tilde{f}^{k}(\tilde{\beta}_{0})$ is on the right of the arc $\tilde{\beta}_{1}$ and of the arc $\gamma \tilde{f}^{k}(\tilde{\beta}_{1})$, the arc $\gamma \tilde{f}^{k}(\tilde{\beta}_{1})$ lies on the right of the arc $\tilde{\beta}_{1}$.

\textbf{Second step.} \emph{ We prove that we can perturb the arc $\tilde{\beta}_{1}$ to obtain an arc $\tilde{\beta}'_{1}$ such that
\begin{enumerate}
\item For any deck transformation $\gamma$ in $a \pi_{1}(S)$ and any integer $\left| k \right| \leq N (l_{\mathcal{G}}(\gamma)+1)$, the curves of the form $\gamma \tilde{f}^{k}(\tilde{\alpha}_{i})$ lie \emph{strictly} on the right of the curve $\tilde{\beta}'_{1}$. 
\item For any deck transformation $\gamma$ in $a \pi_{1}(S)$ and any integer $\left| k \right| \leq N l_{\mathcal{G}}(\gamma)$, the curves of the form $\gamma \tilde{f}^{k}(\tilde{\beta}'_{1})$ lie \emph{strictly} on the right of the curve $\tilde{\beta}'_{1}$.
\end{enumerate}}

Let $M$ be the maximal length with respect to $\mathcal{G}$ of an element $\gamma$ in $a \pi_{1}(S)-a\pi_{1}(S)a^{-1}$ such that there exists $\left| k \right| \leq N l_{\mathcal{G}}(\gamma)$ with $\gamma \tilde{f}^{k}(\tilde{\beta}_{1}) \cap \tilde{\beta}_{1} \neq \emptyset$. As in the proof of Lemma \ref{finitenumber}, one can prove that $M$ is well-defined. Denote by $K$ the compact set consisting of points on $\tilde{\beta}_{1}$ which belong to an arc of the form $\gamma \tilde{f}^{k}(\tilde{\beta}_{1})$, where $l_{\mathcal{G}}(\gamma)=M$ and $\left| k \right| \leq N M$. By maximality of $M$, for any deck transformation $\gamma \in a \pi_{1}(S)-a\pi_{1}(S)a^{-1}$ and any $\left| k \right| \leq N l_{\mathcal{G}}(\gamma)$, the image under $\gamma \tilde{f}^{k}$ of $K$ is sent strictly on the right of $\tilde{\beta}_{1}$. Take a disjoint union $(U_{i})_{i}$ of open disks which cover $K$, such that, for any $i$, $\tilde{\beta}_{1} \cap U_{i}$ is connected and $U_{i}$ is sent strictly to the right of $\tilde{\beta}_{1}$ under a homeomorphism of the form $\gamma \tilde{f}^{k}$, with $\gamma \in a \pi_{1}(S)-a\pi_{1}(S)a^{-1}$ and $\left| k \right| \leq N l_{\mathcal{G}}(\gamma)$. Fix a parameterization of $\tilde{\beta}_{1}$. For any $i$, choose parameters $t_{1,i}<t_{2,i}$ such that $K \cap U_{i} \subset \tilde{\beta}_{1}((t_{1,i},t_{2,i})) \subset U_{i}$. Now, for each $i$, replace in $\tilde{\beta}_{1}$ the portion of arc $\tilde{\beta}_{1}(t_{1,i},t_{2,i})$ with a simple curve $c_{i} : [t_{1,i},t_{2,i}] \rightarrow \overline{U}_{i}$ such that $c_{i}((t_{1,i},t_{2,i}))$ lies strictly on the left of the arc $\tilde{\beta}_{1}$, $c_{i}(t_{1,i})=\tilde{\beta}_{1}(t_{1,i})$ and $c_{i}(t_{2,i})=\tilde{\beta}_{1}(t_{2,i})$. We obtain a new curve $\tilde{\beta}_{1,0}$.

Let us now list the properties of this curve $\tilde{\beta}_{1,0}$. 

First take a deck transformation $\gamma$ in $a \pi_{1}(S)-a \pi_{1}(S)a^{-1}$ with $l_{\mathcal{G}}(\gamma) \geq M$ and $\left| k \right| \leq N l_{\mathcal{G}}(\gamma)$. By construction, the curve $\gamma \tilde{f}^{k}(\tilde{\beta}_{1})$ is strictly on the right of the arc $\tilde{\beta}_{1,0}$. Moreover, by construction of the disks $U_{i}$, the curves $\gamma \tilde{f}^{k}(c_{i})$ are strictly on the right of the arc $\tilde{\beta}_{1}$. Hence these curves lie strictly on the right of the arc $\tilde{\beta}_{1,0}$ which lies on the left of the curve $\tilde{\beta}_{1}$ by construction. We deduce that the arc $\gamma \tilde{f}^{k}(\tilde{\beta}_{1,0})$ lies strictly on the right of the arc $\tilde{\beta}_{1,0}$. Moreover, the curve $\tilde{\beta}_{1,0}$ is strictly on the left of the curve $\gamma \tilde{f}^{k}(\tilde{\beta}_{1})$ as $\gamma^{-1} \notin a \pi_{1}(S)$.

For any deck transformation $\gamma$ in $a \pi_{1}(S) -a \pi_{1}(S)a^{-1}$ and any $\left| k \right| \leq N l_{\mathcal{G}}(\gamma)$, as the curves $\gamma \tilde{f}^{k}(c_{i})$ and the arc $\gamma \tilde{f}^{k}(\tilde{\beta}_{1})$ lie on the right of the arc $\tilde{\beta}_{1}$, the arc $\gamma \tilde{f}^{k}(\tilde{\beta}_{1,0})$ lies on the right of the arc $\tilde{\beta}_{1,0}$, which is itself on the left of $\gamma \tilde{f}^{k}(\tilde{\beta}_{1,0})$.

Take now an automorphism $\gamma \in a\pi_{1}(S)a^{-1}$ and $\left| k \right| \leq N l_{\mathcal{G}}(\gamma)$. Then the arc $\gamma \tilde{f}^{k}(\tilde{\beta}_{1,0})$ lies on the left of the arc $\gamma \tilde{f}^{k}(\tilde{\beta}_{1})$ and the arc $\tilde{\beta}_{1,0}$ lies on the right of the arc $\gamma \tilde{f}^{k}(\tilde{\beta}_{1})$, which is itself on the right of $\tilde{\beta}_{1}$, hence of $\tilde{\beta}_{1,0}$. Hence the arc $\gamma \tilde{f}^{k}(\tilde{\beta}_{1,0})$ lies on the right of the arc $\tilde{\beta}_{1,0}$.

Finally, for any index $1 \leq i \leq b$, any deck transformation $\gamma$ in $a \pi_{1}(S)$ and any integer $\left| k \right| \leq N (l_{\mathcal{G}}(\gamma)+1)$, as the arc $\gamma \tilde{f}^{k}(\tilde{\alpha}_{i})$ lies on the right of the arc $\tilde{\beta}_{1}$, it also lies on the right of the arc $\tilde{\beta}_{1,0}$.

Now repeat the same process with the curve $\tilde{\beta}_{1,0}$ instead of the curve $\tilde{\beta}_{1}$ to obtain a new arc $\tilde{\beta}_{1,1}$ and then repeat it to the arc $\tilde{\beta}_{1,1}$... until we obtain an essential arc $\tilde{\beta}_{1,M}$. This arc satisfies the following properties.
\begin{enumerate}
\item For any deck transformation $\gamma$ in $a \pi_{1}(S)-a \pi_{1}(S)a^{-1}$ and any integer $\left| k \right| \leq N l_{\mathcal{G}}(\gamma)$, the curves of the form $\gamma \tilde{f}^{k}(\tilde{\beta}_{1,M})$ lie strictly on the right of the curve $\tilde{\beta}_{1,M}$.
\item For any deck transformation $\gamma$ in $a \pi_{1}(S)a^{-1}$ and any integer $\left| k \right| \leq N l_{\mathcal{G}}(\gamma)$, the curves of the form $\gamma \tilde{f}^{k}(\tilde{\beta}_{1,M})$ lie on the right of the curve $\tilde{\beta}_{1,M}$.
\item For any deck transformation $\gamma$ in $a \pi_{1}(S)$ and any integer $\left| k \right| \leq N (l_{\mathcal{G}}(\gamma)+1)$, the curves of the form $\gamma \tilde{f}^{k}(\tilde{\alpha}_{i})$ lie on the right of the curve $\tilde{\beta}_{1,M}$. 
\end{enumerate}
Let $\tilde{\beta}'_{1}$ be an essential arc which lies strictly on the left of the arc $\tilde{\beta}_{1,M}$ and is sufficiently close to this arc so that the first above property remains true for the arc $\tilde{\beta}'_{1}$. Then, for any deck transformation $\gamma$ in $a \pi_{1}(S)$ and any integer $\left| k \right| \leq N (l_{\mathcal{G}}(\gamma)+1)$, the curves of the form $\gamma \tilde{f}^{k}(\tilde{\alpha}_{i})$ lie strictly on the right of the curve $\tilde{\beta}_{1}$. If we take an automorphism $\gamma$ in $a \pi_{1}(S)a^{-1}$ and an integer $\left| k \right| \leq N l_{\mathcal{G}}(\gamma)$, observe that the curve $\gamma \tilde{f}^{k}(\tilde{\beta}'_{1})$ lies strictly on the left of the curve $\gamma \tilde{f}^{k}(\tilde{\beta}_{1,M})$ which lies on the right of the curve $\tilde{\beta}_{1,M}$ and that the arc $\tilde{\beta}_{1,M}$ lies on the right of the arc $\gamma \tilde{f}^{k}(\tilde{\beta}_{1,M})$. Hence, the arc $\gamma \tilde{f}^{k}(\tilde{\beta}'_{1})$ lies strictly on the right of the arc $\tilde{\beta}'_{1}$.

Note that the curves of the form $\gamma \tilde{f}^{k}(\tilde{\beta}'_{1})$, where $\gamma \in \pi_{1}(S)a^{-1}- a \pi_{1}(S)a^{-1}$ and $\left|k \right| \leq N l_{\mathcal{G}}(\gamma)$, lie strictly on the left of the curve $\tilde{\beta}'_{1}$, as $\gamma \tilde{f}^{k}(\tilde{\beta}'_{1}) \cap \tilde{\beta}'_{1}= \gamma \tilde{f}^{k}(\tilde{\beta}'_{1} \cap \gamma^{-1} \tilde{f}^{-k}(\tilde{\beta}'_{1}))=\emptyset$ and $\gamma^{-1} \notin a \pi_{1}(S)$. After this second step, it is still possible that the curve $\tilde{\beta}'_{1}$ meets a curve of the form $\gamma \tilde{f}^{k} (\tilde{\beta}'_{1})$, with $\gamma \in \Gamma$ and $\left| k \right| \leq N l_{\mathcal{G}}(\gamma)$.

\textbf{Third step.} We finally construct the curve $\tilde{\beta}_{2}$ with the properties required by Lemma \ref{goodcurve}. To achieve this, we construct by induction a sequence of curves. Let $M'$ be the maximal length (with respect to $\mathcal{G}$) of an element $\gamma$ in $\Gamma \cup \pi_{1}(S) a^{-1} -a \pi_{1}(S) a^{-1}$ such that either there exist $\left| k \right| \leq N(l_{\mathcal{G}}(\gamma)+1)$ and $1 \leq i \leq b$ such that $\gamma \tilde{f}^{k}(\tilde{\alpha}_{i}) \cap \tilde{\beta}'_{1} \neq \emptyset$ or there exists $\left| k \right| \leq Nl_{\mathcal{G}}(\gamma)$ such that $\gamma \tilde{f}^{k}(\tilde{\beta}'_{1}) \cap \tilde{\beta}'_{1} \neq \emptyset$. One can prove that this maximum is well-defined by an argument similar to the proof of Lemma \ref{finitenumber}: otherwise there would be a contradiction with the hypothesis $\frac{\mathrm{diam}(\tilde{f}^{n}(D))}{n} \xrightarrow[n \to + \infty]{} 0$ for any fundamental domain $D \subset \tilde{S}$ for the action of the group $\pi_{1}(S)$.

Let $\tilde{\delta}_{M'+1}= \tilde{\beta}'_{1}$. We now construct by induction on $j \in [0,M+1]$ a curve $\tilde{\delta}_{j}$ whose endpoints $\tilde{\delta}_{j}(0)$ and $\tilde{\delta}_{j}(1)$ lie on the same components of $\partial \tilde{S}$ as the points $\tilde{\beta}_{0}(0)$ and $\tilde{\beta}_{0}(1)$ respectively with the following properties.
\begin{enumerate}
\item For any element $\gamma$ in $a \pi_{1}(S)$ and any integer $\left| k \right| \leq N (l_{\mathcal{G}}(\gamma)+1)$, the curves of the form $\gamma \tilde{f}^{k}(\tilde{\alpha}_{i})$ lie strictly on the right of the curve $\tilde{\delta}_{j}$. 
\item For any element $\gamma$ in $a \pi_{1}(S)$ and any integer $\left| k \right| \leq N l_{\mathcal{G}}(\gamma)$, the curves of the form $\gamma \tilde{f}^{k}(\tilde{\delta}_{j})$ lie strictly on the right of the curve $\tilde{\delta}_{j}$.
\item For any element $\gamma$ in $\Gamma \cup \pi_{1}(S) a^{-1} - a \pi_{1}(S) a^{-1}$ with $l_{\mathcal{G}}(\gamma) \geq j$ and any integer $\left| k \right| \leq N (l_{\mathcal{G}}(\gamma)+1)$, the curves of the form $\gamma \tilde{f}^{k}(\tilde{\alpha}_{i})$ lie strictly on the left of the curve $\tilde{\delta}_{j}$.
\item For any non-trivial element $\gamma$ in $\Gamma$ with $l_{\mathcal{G}}(\gamma) \geq j$ and any integer $\left| k \right| \leq N l_{\mathcal{G}}(\gamma)$, the curves of the form $\gamma \tilde{f}^{k}(\tilde{\delta}_{j})$ lie strictly on the left of the curve $\tilde{\delta}_{j}$.
\end{enumerate}
Then it suffices to take $\tilde{\beta}_{2}= \tilde{\delta}_{0}$ to complete the proof of Lemma \ref{goodcurve}.

Suppose that we have constructed an arc $\tilde{\delta}_{j+1}$ for $j \geq 0$ with the above properties. Let us build the arc $\tilde{\delta}_{j}$. Denote by $\mathcal{C}'$ the set of connected components of $\partial \tilde{S}$ which lie on the right of $\tilde{\delta}_{j+1}$ (or equivalently on the right of $\tilde{\beta}_{0}$). 
For any $1 \leq i \leq b$, denote by $B_{i}$ the subset of $\Gamma \cup \pi_{1}(S)a^{-1}- a \pi_{1}(S)a^{-1}$ consisting of deck transformations $\gamma$ such that $D_{b}$ lies on the left of $\gamma(\tilde{\alpha}_{i})$. Denote by $B_{i}^{c}$ the complement of this set in $\Gamma \cup \pi_{1}(S)a^{-1}- a \pi_{1}(S)a^{-1}$. Denote by $\mathcal{F}_{j}$ the family of essential arcs consisting of the following arcs.
\begin{enumerate}
\item The arcs of the form $\gamma\tilde{f}^{k}(\tilde{\alpha}_{i})$ where $1 \leq i \leq b$, $\left|k \right| \leq N(l_{\mathcal{G}}(\gamma)+1)$, $\gamma \in B_{i}^{c}$ and $l_{\mathcal{G}}(\gamma)=j$.
\item The arcs of the form $\gamma\tilde{f}^{k}(\tilde{\alpha}_{i}^{0})$ where $1 \leq i \leq b$, $\left|k \right| \leq N(l_{\mathcal{G}}(\gamma)+1)$, $\gamma \in B_{i}$ and $l_{\mathcal{G}}(\gamma)=j$.
\item The arcs of the form $\gamma \tilde{f}^{k}(\tilde{\delta}_{j+1}^{0})$ where $\left|k \right| \leq Nl_{\mathcal{G}}(\gamma)$, $\gamma \in \Gamma-\left\{ Id \right\}$ and $l_{\mathcal{G}}(\gamma)=j$.
\end{enumerate}

We want to define $\tilde{\delta}'_{j}= \sup_{\tilde{\delta}_{j+1}}(\mathcal{F}_{j})$. To do so, we first have to check that the family $\mathcal{F}_{j}$ satisfies the hypothesis of Lemma \ref{inf} (or more precisely of the lemma obtained from Lemma \ref{inf} by changing the word "left" with the word "right", see the remark below Lemma \ref{inf}).  By definition of the sets $B_{i}$, any component in $\mathcal{C}'$ lies on the left of any arc of the form $\gamma(\tilde{\alpha}_{i})$, with $\gamma \in B_{i}$, hence also of the arcs of the form $\tilde{f}^{k} \gamma(\tilde{\alpha}_{i})$ for any $k \in \mathbb{Z}$. Moreover, any component in $\mathcal{C}'$ lies on the right of any arc of the form $\tilde{f}^{k} \gamma(\tilde{\alpha}_{i})$, with $\gamma \in B_{i}^{c}$ and $k \in \mathbb{Z}$. Finally, take any $\gamma \in \Gamma$. Remember that the curve $\gamma^{-1}(\tilde{\beta}_{0})$ lies strictly on the left of the curve $\tilde{\beta}_{0}$, as $\gamma^{-1} \notin a\pi_{1}(S)$. Hence the arc $\tilde{\beta}_{0}$ lies strictly on the left of the curve $\gamma(\tilde{\beta}_{0})$. Therefore any component in $\mathcal{C}'$ lies on the left of $\gamma(\tilde{\beta}_{0})$ and also on the left of $\gamma \tilde{f}^{k}(\tilde{\delta}_{j+1})$ for any $k \in \mathbb{Z}$: the connected components of $\partial \tilde{S}$ met by the curve $\gamma(\tilde{\beta}_{0})$ are the same as those met by the curve $\gamma \tilde{f}^{k}(\tilde{\delta}_{j+1})$ and these curves are oriented in the same way.

We can apply Lemma \ref{inf} to obtain an essential arc $\tilde{\delta}'_{j}= \sup_{\tilde{\delta}_{j+1}}(\mathcal{F}_{j})$.

Let us study the properties of this curve.

First let us check that, for any element $\gamma$ in $a \pi_{1}(S)$ and any integer $\left| k \right| \leq N (l_{\mathcal{G}}(\gamma)+1)$, the curves of the form $\gamma \tilde{f}^{k}(\tilde{\alpha}_{i})$ lie strictly on the right of the curve $\tilde{\delta}'_{j}$. Fix such an element $\gamma_{0}$ in $a \pi_{1}(S)$, such an integer $k_{0}$ and $1 \leq i_{0} \leq b$. As one of the endpoints of $\gamma_{0} \tilde{f}^{k_{0}}(\tilde{\alpha}_{i_{0}})$ lies on a connected component of $\partial \tilde{S}$ which is strictly on the right of $\tilde{\beta}_{1}$ (hence of $\tilde{\delta}'_{j}$), it suffices to prove that $\gamma_{0} \tilde{f}^{k_{0}}(\tilde{\alpha}_{i_{0}}) \cap \tilde{\delta}'_{j}= \emptyset$. By definition of the arc $\tilde{\delta}'_{j}$, it suffices to prove that the arc $\gamma_{0} \tilde{f}^{k_{0}}(\tilde{\alpha}_{i_{0}})$ lies strictly on the right of the arc $\tilde{\delta}_{j+1}$ and of any arc of $\mathcal{F}_{j}$. By induction hypothesis, the arc $\gamma_{0} \tilde{f}^{k_{0}}(\tilde{\alpha}_{i_{0}})$ lies strictly on the right of the arc $\tilde{\delta}_{j+1}$. It also lies strictly on the right of the curves of the form $\gamma\tilde{f}^{k}(\tilde{\alpha}_{i})$ where $1 \leq i \leq b$, $\left|k \right| \leq N(l_{\mathcal{G}}(\gamma)+1)$, $\gamma \in B_{i}^{c}$ and $l_{\mathcal{G}}(\gamma)=j$: the curve $\gamma^{-1} \gamma_{0} \tilde{f}^{k_{0}-k}(\tilde{\alpha}_{i_{0}})$ is disjoint from the curve $\tilde{\alpha}_{i}$, as 
$$\left|k_{0}-k \right| \leq N(l_{\mathcal{G}}(\gamma)+l_{\mathcal{G}}(\gamma_{0})+2)=N(l_{\mathcal{G}}(\gamma^{-1}\gamma_{0})+2) \leq 3N l_{\mathcal{G}}(\gamma^{-1}\gamma_{0}),$$
by hypothesis on the curves $\tilde{\alpha}_{j}$ (recall that these curves satisfy $\mathcal{P}(3N)$, see the beginning of the proof of Proposition \ref{disjointcurvesboundary}). For the same reason, the arc $\gamma_{0} \tilde{f}^{k_{0}}(\tilde{\alpha}_{i_{0}})$ lies strictly on the right of the arcs of the form $\gamma\tilde{f}^{k}(\tilde{\alpha}_{i}^{0})$ where $1 \leq i \leq b$, $\left|k \right| \leq N(l_{\mathcal{G}}(\gamma)+1)$, $\gamma \in B_{i}$ and $l_{\mathcal{G}}(\gamma)=j$. Finally, let us prove that the arc $\gamma_{0} \tilde{f}^{k_{0}}(\tilde{\alpha}_{i_{0}})$ lies strictly on the right of the arcs of the form $\gamma \tilde{f}^{k}(\tilde{\delta}_{j+1}^{0})$ where $\left|k \right| \leq Nl_{\mathcal{G}}(\gamma)$, $\gamma \in \Gamma-\left\{ Id \right\}$ and $l_{\mathcal{G}}(\gamma)=j$. Notice that the deck transformation $\gamma^{-1} \gamma_{0}$ belongs either to $\Gamma$ or to $\pi_{1}(S)a^{-1}- a \pi_{1}(S) a^{-1}$ and that $l_{\mathcal{G}}(\gamma^{-1}\gamma_{0})=l_{\mathcal{G}}(\gamma)+l_{\mathcal{G}}(\gamma_{0})>j$. In both cases, the claim is true as the arc $\gamma^{-1}\gamma_{0} \tilde{f}^{k_{0}-k}(\tilde{\alpha}_{i_{0}})$ is disjoint from $\tilde{\delta}_{j+1}$ by induction hypothesis.

Let us see why, for any element $\gamma$ in $a \pi_{1}(S)$ and any integer $\left| k \right| \leq N l_{\mathcal{G}}(\gamma)$, the curves of the form $\gamma \tilde{f}^{k}(\tilde{\delta}'_{j})$ lie strictly on the right of the curve $\tilde{\delta}'_{j}$. Fix such an element $\gamma_{0}$ in $a \pi_{1}(S)$ and such an integer $k_{0}$. Here we distinguish the cases $\gamma_{0} \in a \pi_{1}(S)-a \pi_{1}(S)a^{-1}$ and $\gamma_{0} \in a \pi_{1}(S)a^{-1}$. In the first case, notice that the curve $\gamma_{0} \tilde{f}^{k_{0}}(\tilde{\delta}'_{j})$ is on the right of the curve $\gamma_{0} \tilde{f}^{k_{0}}(\tilde{\delta}_{j+1})$ by definition of $\tilde{\delta}'_{j}$. Hence it suffices to prove that this last arc is strictly on the right of any arc in $\mathcal{F}_{j}\cup \left\{ \tilde{\delta}_{j+1} \right\}$. To do this, it suffices to prove that any arc in $\mathcal{F}_{j} \cup \left\{ \tilde{\delta}_{j+1} \right\}$ is on the left of the arc $\gamma_{0} \tilde{f}^{k_{0}}(\tilde{\delta}_{j+1})$, which is easily done by using the induction hypothesis. Now suppose that $\gamma_{0} \in a \pi_{1}(S)a^{-1}$. As usual, as one of the endpoints of the arc $\gamma_{0} \tilde{f}^{k_{0}}(\tilde{\delta}_{j})$ is strictly on the right of the curve $\tilde{\delta}_{j}$, it suffices to prove that $\gamma_{0} \tilde{f}^{k_{0}}(\tilde{\delta}_{j}) \cap \tilde{\delta}_{j}= \emptyset$. To do this, it suffices to check that the image under $\gamma_{0}\tilde{f}^{k_{0}}$ of any essential arc in $\mathcal{F}_{j}\cup \left\{ \tilde{\delta}_{j+1} \right\}$ is disjoint from any arc in $\mathcal{F}_{j} \cup \left\{ \tilde{\delta}_{j+1} \right\}$ which can be done without serious difficulty by using the induction hypothesis and the properties of the curves $\tilde{\alpha}_{i}$.

By construction, the curves of the form $\gamma \tilde{f}^{k}(\tilde{\alpha}_{i})$ with $1 \leq i \leq b$, $\gamma \in \Gamma \cup \pi_{1}(S)a^{-1}- a \pi_{1}(S)a^{-1}$, $l_{\mathcal{G}}(\gamma)=j$ lie on the left of the arc $\tilde{\delta}'_{j}$. By induction hypothesis, as the arc $\tilde{\delta}_{j+1}$ lies on the left of the arc $\tilde{\delta}'_{j}$, the curves of the form $\gamma \tilde{f}^{k}(\tilde{\alpha}_{i})$ with $1 \leq i \leq b$, $\gamma \in \Gamma \cup \pi_{1}(S)a^{-1}- a \pi_{1}(S)a^{-1}$, $l_{\mathcal{G}}(\gamma)>j$ lie strictly on the left of the arc $\delta'_{j}$.

Finally, let us check that, for any non-trivial element $\gamma$ in $\Gamma$ with $l_{\mathcal{G}}(\gamma) \geq j$ and any integer $\left| k \right| \leq N l_{\mathcal{G}}(\gamma)$, the curves of the form $\gamma \tilde{f}^{k}(\tilde{\delta}'_{j})$ lie on the left of the curve $\tilde{\delta}'_{j}$. Fix such an element $\gamma_{0}\in \Gamma$ and such an integer $k_{0}$. This results from the following facts.
\begin{enumerate}
\item The arc $\gamma_{0} \tilde{f}^{k_{0}}(\tilde{\delta}'_{j})$ lies on the right of the arc $\gamma_{0} \tilde{f}^{k_{0}}(\tilde{\delta}_{j+1})$.
\item The arc $\tilde{\delta}'_{j}$ lies on the left of the arc $\gamma_{0} \tilde{f}^{k_{0}}(\tilde{\delta}_{j+1})$.
\item The arc $\gamma_{0} \tilde{f}^{k_{0}}(\tilde{\delta}_{j+1})$ is on the left of the arc $\tilde{\delta}'_{j}$.
\end{enumerate}

With arguments similar to those used during Step 2, we then perturb the arc $\tilde{\delta}'_{j}$ to obtain an arc $\tilde{\delta}_{j}$ which satisfies the required properties. 
\end{proof}

\end{document}